\documentclass[1 [leqno,11pt]{amsart}
\usepackage{amssymb, amsmath}

\usepackage{pdfsync}
\newcommand{\andf}{\quad\hbox{and}\quad}
\newcommand{\with}{\quad\hbox{with}\quad}
\def\Supp{\mathop{\rm Supp}\nolimits\ }
\newcommand{\newcom}{\newcommand}
\def\longformule#1#2{
\displaylines{ \qquad{#1} \hfill\cr \hfill {#2} \qquad\cr } }
\def\inte#1{
\displaystyle\mathop{#1\kern0pt}^\circ }
\newcommand{\w}[1]{\langle {#1} \rangle}
\newcom{\al}{\alpha}
\newcom{\de}{\delta}
\newcom{\be}{\beta}
\newcom{\s}{\sigma}
\newcom{\eps}{\epsilon}
\newcom{\ve}{\varepsilon}
\newcom{\ga}{\gamma}
\newcom{\Ga}{\Gamma}
\newcom{\ka}{\kappa}
\newcom{\Lam}{\Lambda}
\newcom{\lam}{\lambda}
\newcom{\vp}{\varphi}
\newcom{\om}{\omega}
\newcom{\Sig}{\Sigma}
\newcom{\sig}{\sigma}
\newcom{\tht}{\theta}
\newcom{\tri}{\triangle}
\newcom{\oo}{\infty}
\newcom{\h}{{\rm h}}
\newcom{\vphi}{\varphi}
\newcom{\cB}{{\mathcal B}}
\newcom{\cC}{{\mathcal C}}
\newcom{\cD}{{\mathcal D}}
\newcom{\cF}{{\mathcal F}}
\newcom{\cL}{{\mathcal L}}
\newcom{\cM}{{\mathcal M}}
\newcom{\cP}{{\mathcal P}}
\newcom{\cS}{{\mathcal S}}
\newcom{\cQ}{{\mathcal Q}}
\newcom{\cT}{{\mathcal T}}
\newcom{\cY}{{\mathcal Y}}
\newcom{\cZ}{{\mathcal Z}}
\newcom{\R}{\Bbb R}
\newcom{\T}{\Bbb T}
\newcom{\N}{\Bbb N}
\newcom{\Z}{\Bbb Z}
\newcom{\C}{\Bbb C}
\newcom{\E}{\Bbb E}
\let\wh=\widehat

\let\e=\varepsilon

\newcom{\f}{\frac}
\newcom{\dint}{\displaystyle\int}
\newcom{\dsum}{\displaystyle\sum}
\newcom{\dlim}{\displaystyle\lim}
\newcom{\ov}{\overline}
\newcom{\wt}{\widetilde}
\newcom{\pa}{\partial}
\newcom{\p}{\partial}
\newcom\na{\nabla}
\newcom{\D}{\Delta}
\def\vf{\varphi}
\newcom\rto{\rightarrow}
\newcom\lto{\leftarrow}
\newcom\mto{\mapsto}
\newcom{\disp}{\displaystyle}
\newcom{\non}{\nonumber}
\newcom{\no}{\noindent}
\newcom{\QED}{$\square$}

\def\eqdefa{\buildrel\hbox{\footnotesize def}\over =}

\newcommand{\beq}{\begin{equation}}
\newcommand{\eeq}{\end{equation}}
\newcommand{\ben}{\begin{eqnarray}}
\newcommand{\een}{\end{eqnarray}}
\newcommand{\beno}{\begin{eqnarray*}}
\newcommand{\eeno}{\end{eqnarray*}}

\newtheorem{Def}{Definition}[section]
\newtheorem{thm}{Theorem}[section]
\newtheorem{lem}{Lemma}[section]
\newtheorem{rmk}{Remark}[section]
\newtheorem{col}{Corollary}[section]
\newtheorem{prop}{Proposition}[section]
\renewcommand{\theequation}{\thesection.\arabic{equation}}

\begin{document}
\title[global well-posedness of Prandtl system]
{Global existence and decay of solutions to Prandtl system with small analytic data}

\author{Marius Paicu}
\address{Universit\'e  Bordeaux \\
 Institut de Math\'ematiques de Bordeaux\\
F-33405 Talence Cedex, France} \email{marius.paicu@math.u-bordeaux.fr}

\author{Ping Zhang}%
\address{Academy of
Mathematics $\&$ Systems Science and  Hua Loo-Keng Key Laboratory of
Mathematics, The Chinese Academy of Sciences\\
Beijing 100190, CHINA } \email{zp@amss.ac.cn}

\date{\today}

\begin{abstract} In this paper, we prove the
 global existence and the large time decay estimate  of solutions to  Prandtl system with small initial data, which is analytical  in the tangential variable.
 The key ingredient used in the proof  is to derive sufficiently fast decay-in-time estimate of some weighted analytic energy estimate
 to a quantity, which consists of a linear combination of the tangential velocity  with its primitive one, and
 which basically controls the evolution of the analytical radius to the solutions.
  Our result can be viewed  as a global-in-time Cauchy-Kowalevsakya result for  Prandtl system with small analytical data.
\end{abstract}

\maketitle

\vspace{0.1cm}


\renewcommand{\theequation}{\thesection.\arabic{equation}}
\setcounter{equation}{0}
\section{Introduction}\label{sect1}

Describing the behavior of boundary layers is one of the most  challenging and  important problem in the mathematical  fluid mechanics. The governing equations of the boundary layer obtained by vanishing viscosity   of Navier-Stokes system with Dirichlet boundary condition, was proposed by Prandtl \cite{Pra} in 1904 in order to explain the disparity between the boundary conditions verified by  ideal fluid and viscous fluid with small viscosity.
Heuristically, these
 boundary layers are of amplitude $O(1)$ and of  thickness  $O(\sqrt{\nu})$ where there is a transition from the interior flow governed by
 Euler equation to the Navier-Stokes flow with a vanishing viscosity $\nu> 0$.  One may check \cite{E, Olei} and
references therein for more introductions on boundary layer theory. Especially we refer to \cite{Guo} for a comprehensive  recent survey.
\smallskip

 One of the key step to rigorously justify this inviscid limit of
Navier-Stokes system with Dirichelt boundary condition is to deal
with the well-posedness of the following Prandtl system,
\begin{equation}\label{S1eq1}
\left\{
\begin{array}{ll}
\p_tU+U\pa_x U+V\p_yU-\p_{y}^2U+\pa_xp=0, \quad (t,x,y)\in\R_+\times\R\times\R_+, \\
\pa_xU+\p_yV =0,\\
U|_{y=0}=V|_{y=0}=0 \andf \lim_{y\to +\infty}U(t,x,y)=w(t,x),\\
U|_{t=0}=U_{0},
\end{array}
\right.
\end{equation}
where  $U$ and $V$  represent
the tangential and normal velocities of the boundary layer flow.  $(w(t,x),p(t,x))$ are
the traces of the tangential velocity and pressure of the outflow on the boundary, which satisfy
Bernoulli's law:
\beq \label{S1eq1a}
\p_tw+w\p_xw+\p_xp=0.
\eeq

 Since there is no
horizontal diffusion in the $U$ equation of \eqref{S1eq1}, the
nonlinear term $V\p_y U$ (which almost behaves like $-\p_xU\p_y U$)
loses one horizontal derivative in the process of energy estimate,
and therefore the question of whether or not the Prandtl system with
general data is well-posed in Sobolev spaces is still open. In fact, E and Enquist \cite{EE} constructed
a class of initial data which generate solutions with finite time singularities in case the solutions exist locally in time.
G\'{e}rard-Varet and Dormy \cite{Ger1} proved the
ill-posedness in Sobloev spaces for the linearized Prandtl system
around non-monotonic shear flows. The nonlinear ill-posedness was
also established in \cite{Ger2, Guo} in the sense of non-Lipschtiz
continuity of the flow. Nevertheless, we have the following positive
results for two classes of special data.

$\bullet$ Under a monotonic assumption on the tangential velocity of
the outflow, Oleinik \cite{Olei} first introduced Crocco transformation and then proved the local
existence and uniqueness of classical solutions to \eqref{S1eq1}. With
the additional ``favorable" condition on the pressure, Xin and Zhang
\cite{Xin} obtained the global existence of weak solutions to this
system. Recently, by ingenious use of the cancelation  property of  the bad terms containing the tangential derivative,
  the authors of \cite{Alex} and  \cite{MW} succeeded in  proving the existence of local smooth solution  to \eqref{S1eq1} in Sobolev space
  via performing energy estimates in weighted Sobolev spaces.

$\bullet$ For the data which is analytic in both $x$ and $y$ variables,
Sammartino and Caflisch \cite{Caf} established the local
well-posedness result of \eqref{S1eq1}. The analyticity in $y$
variable was removed by Lombardo, Cannone and Sammartino in
\cite{Can}. The main argument used in  \cite{Can, Caf} is to apply
the abstract Cauchy-Kowalewskaya (CK) theorem. Lately, G\'ervard-Varet and Masmoudi \cite{GM} proved the
 well-posedness of \eqref{S1eq1} for a class of data with
Gevrey regularity. This result was improved to be optimal in sense of \cite{Ger1} in \cite{DG08} by Dietert and G\'ervard-Varet.
 The question of the long time existence for Prandtl system with small analytic data was first addressed in \cite{ZHZ} and an almost global existence result was provided in \cite{IV16}. \smallskip

In this paper, we investigate the global existence and the large time decay estimates of the solutions to  Prandtl  system with small data which is analytic in the tangential variable.
For simplicity, here we take $w(t,x)$ in \eqref{S1eq1} to be $\e f(t)$ with $f(0)=0,$ which along with \eqref{S1eq1a} implies $\p_xp=-\e f'(t).$ Let us take a cut-off function
$\chi\in C^\infty[0,\infty)$ with $\chi(y)=\left\{
\begin{array}{ll}
1\quad \mbox{if} \  y\geq 2 \\
0\quad  \mbox{if} \  y\leq 1,
\end{array}
\right.$
 we denote $W\eqdefa U-\e f(t)\chi(y).$
 Then $W$ solves
\beq \label{S1eq3}
\left\{
\begin{array}{ll}
\p_tW+\left(W+\e f(t) \chi(y)\right)\pa_x W+V\p_y\left(W+\e   f(t) \chi(y)\right)-\p_{y}^2W=\e m,\\
\pa_xW+\p_yV =0, \quad (t,x,y)\in\R_+\times\R^2_+,\\
W|_{y=0}=V|_{y=0}=0 \andf \lim_{y\to +\infty}W(t,x,y)=0,\\
W|_{t=0}=U_{0},
\end{array}
\right.
\end{equation}
where $\R^2_+\eqdefa\R\times\R_+$ and $m(t,y)\eqdefa (1-\chi(y))f'(t)+ f(t)\chi''(y).$

In order to get rid of the source term in the $W$ equation of \eqref{S1eq3}, we introduce
 $u^s$ via
\begin{equation}\label{S1eq4}
\left\{
\begin{array}{ll}
\p_tu^s-\p_y^2u^s=\e m(t,y), \quad (t,y)\in\R_+\times\R_+, \\
u^s|_{y=0}=0 \andf \lim_{y\to +\infty} u^s(t,y)=0,\\
u^s|_{t=0}=0.
\end{array}
\right.
\end{equation}
With $u^s$ being determined by \eqref{S1eq4}, we set $u\eqdefa W- u^s$ and $v\eqdefa V.$ Then $(u,v)$ verifies
\beq \label{S1eq5}
\left\{
\begin{array}{ll}
\p_tu+\left(u+u^s+\e f(t) \chi(y) \right)\pa_x u+v\p_y\left(u+u^s+\e f(t)\chi(y)\right)-\p_{y}^2u=0, \\
\pa_xu+\p_yv =0, \quad (t,x,y)\in\R_+\times\R^2_+,\\
u|_{y=0}=v|_{y=0}=0 \andf \lim_{y\to \infty}u(t,x,y)=0,\\
u|_{t=0}=u_0\eqdefa U_{0}.
\end{array}
\right.
\end{equation}

On the other hand,
due to $\p_xu+\p_yv=0,$ there exists a potential function $\vf$ so that $u=\p_y\vf$ and
$v=-\p_x\vf.$ Then by integrating the $u$ equation of \eqref{S1eq5} with respect to
$y$ variable over $[y,\infty),$ we obtain
\beno
\begin{split}
&\p_t\vf+\left(u+u^s+\e f(t)\chi(y)\right)\p_x\vf\\
&\qquad+2\int_y^\infty\left(\p_y \left(u+u^s+\e f(t)\chi(y')\right)\p_x\vf\right)\,dy'-\p_y^2\vf=Q(t,x),
\end{split}
\eeno for some function $Q(t,x).$ Yet since we assume that $\vf$ decays to zero sufficiently fast as $y$ approaching to $+\infty,$
we find that $Q(t,x)=0.$  Therefore, by virtue of \eqref{S1eq5}, $\vf$ satisfies
\beq \label{S1eq2}
\left\{
\begin{array}{ll}
\p_t\vf+\left(u+u^s+\e f(t)\chi(y)\right)\p_x\vf\\
\qquad+2\int_y^\infty\left(\p_y \left(u+u^s+\e f(t)\chi(y')\right)\p_x\vf\right)\,dy'-\p_y^2\vf=0,\\
\p_y\vf|_{y=0}=0 \andf \lim_{y\to +\infty}\vf(t,x,y)=0,\\
\vf|_{t=0}=\vf_0.
\end{array}
\right.
\eeq

In order to globally control the evolution of the analytic band to the solutions of \eqref{S1eq5}, we introduce
the following key quantity:
\beq \label{S7eq4}
G\eqdefa u+\f{y}{2\w{t}}\vf \andf g\eqdefa \p_yG= \p_yu+\f{y}{2\w{t}}u+\f\vf{2\w{t}}.
\eeq
 We emphasize  that  the introduction of those quantities $G$ and $g$  in \eqref{S7eq4} is in fact inspired by the function $\frak{g}\eqdefa \p_yu+\f{y}{2\w{t}}u,$
which was introduced by Ignatova and Vicol in \cite{IV16}, where the authors of \cite{IV16} basically
proved that the  weighted analytical norm of $\frak{g}(t)$ decays like $\w{t}^{-\left(\f54\right)_-},$ which decays
faster than the  weighted analytical norm of $u$ itself. We observe that  $g-\frak{g}=\f\vf{2\w{t}}.$ One novelty of this
paper is to prove that the analytical norm of $g$ is almost decays like $\w{t}^{-\f74}.$

At the beginning of Section \ref{Sect7}, we shall show that $G$ verifies
\beq \label{S7eq6}
\left\{
\begin{array}{ll}
\p_tG-\p_y^2G+\w{t}^{-1}G+\left(u+u^s+\e f(t)\chi(y)\right)\p_xG+v\p_yG\\
\qquad+v\p_y\left(u^s+\e f(t)\chi(y)\right)-\f12\w{t}^{-1}v\p_y({y\vf})\\
\qquad+\f{y}{\w{t}}\int_y^\infty\left(\p_y \left(u+u^s+\e f(t)\chi(y')\right)\p_x\vf\right)\,dy'=0,\\
G|_{y=0}=0 \andf \lim_{y\to +\infty}G(t,x,y)=0,\\
G|_{t=0}=G_0\eqdefa u_0+\f{y}2\vf_0.
\end{array}
\right.
\eeq

The main result of this paper states as follows:

\begin{thm}\label{th1.1}
{\sl Let  $\delta>0$  and $f\in H^1(\R_+)$ which satisfies
\beq \label{S2eq12}
\cC_{f}\eqdefa \int_{0}^\infty\w{t}^{\f54}\bigl(|f(t)|+|f'(t)|\bigr)\,dt + \Bigl(\int_{0}^\infty\w{t}^{\f72}\bigl(f^2(t)+(f'(t))^2\bigr)\,dt\Bigr)^{\f12}<\infty.
\eeq
 Let  $u_0=\p_y\vf_0$ satisfy $u_0(x,0)=0,$  $\int_0^\infty u_0\,dy=0$ and $\bigl\|e^{\f{y^2}{8}}e^{\delta|D_x|}(\varphi_0,u_0)\bigr\|_{\cB^{\f12,0}} <\infty.$ We assume moreover that
  $G_0= u_0+\f{y}2\vf_0$ satisfies
\beq\label{S1eq6}
\bigl\|e^{\f{y^2}{8}}e^{\delta|D_x|} G_0\bigr\|_{\cB^{\f12,0}}\leq c_0
\eeq
 for some $c_0$ sufficiently small. Then \eqref{S1eq4} has a solution $u^s$ and
  there exists $\e_0>0$ so that  for $\varepsilon\leq \e_0$, the system (\ref{S1eq5}) has a  unique global solution $u$ which satisfies
\beq\label{S1eq7}
\bigl\| e^{\f{y^2}{8\w{t}}} e^{\frac{\delta}{2}|D_x|}u \bigr\|_{{L}^\infty(\R_+;\cB^{\f12,0})}+\bigl\|e^{\f{y^2}{8\w{t}}} e^{\frac{\delta}{2}|D_x|}\p_y u\bigr\|_{{L}^2(\R_+;\cB^{\f12,0})}\leq C\bigl\|e^{\f{y^2}{8}}e^{\delta |D_x|}u_0\bigr\|_{\cB^{\f12,0}}.
\eeq
Furthermore,  for any $t>0,$   there hold
\beq \label{S1eq10}
\begin{split}
&\bigl\|\w{t}^{\f34} e^{\f{y^2}{8\w{t}}} e^{\frac{\delta}{2}|D_x|}u(t) \bigr\|_{\cB^{\f12,0}}
+\bigl\|\w{t'}^{\f34} e^{\f{y^2}{8\w{t}}} e^{\frac{\delta}{2}|D_x|}\p_yu \bigr\|_{L^2(t/2,t;\cB^{\f12,0})}\\
&\qquad\qquad\qquad\qquad\qquad\qquad\qquad\qquad\qquad\qquad\leq C\|e^{\f{y^2}{8}}e^{\delta|D_x|}(\varphi_0,u_0)\|_{\cB^{\f12,0}},\\
&\bigl\|\w{t}^{\f54} e^{\f{y^2}{8\w{t}}} e^{\frac{\delta}{2}|D_x|}G(t) \bigr\|_{\cB^{\f12,0}}
+\bigl\|\w{t'}^{\f54} e^{\f{y^2}{8\w{t}}} e^{\frac{\delta}{2}|D_x|}\p_yG \bigr\|_{L^2(t/2,t;\cB^{\f12,0})}
\leq C\|e^{\f{y^2}{8}}e^{\delta|D_x|}G_0\|_{\cB^{\f12,0}},
\end{split}
\eeq
and
\beq \label{S1eq10bn}
\begin{split}
\bigl\|\w{t}^{\f54} e^{\f{\ga y^2}{8\w{t}}} e^{\frac{\delta}{2}|D_x|}u(t) \bigr\|_{\cB^{\f12,0}}
+\bigl\|\w{t'}^{\f54} e^{\f{\ga y^2}{8\w{t}}} e^{\frac{\delta}{2}|D_x|}\p_yu \bigr\|_{L^2(t/2,t;\cB^{\f12,0})}
\leq C\|e^{\f{y^2}{8}}e^{\delta|D_x|}G_0\|_{\cB^{\f12,0}},
\end{split}
\eeq  for any $\ga\in (0,1).$
}
\end{thm}

The anisotropic Besov spaces $\cB^{\f12,0}$ will be recalled in Section
\ref{sect2}. Here and all in that follows, we always denote $\w{t}\eqdefa 1+t.$

\begin{rmk}
\begin{itemize}

\item[(1)]
 In  the previous results concerning the long time well-posedness of the Prandtl system in \cite{IV16, ZHZ},  only a lower bound of the lifespan to the solution was obtained. We also mention  that similar type of result as in \cite{IV16, ZHZ}  for the lifespan of MHD boundary layer equation was obtained in \cite{XY19}.

\item[(2)] Our global well-posedness result does not contradict with the blow-up result in \cite{EE}. In fact,
 Theorem 1.1 of \cite{EE} claims that if $u_0(0,y)=0$ and $a_0(y)=-\p_xu_0(0,y)$   is nonnegative
and of compact support such that
\beq \label{S1eq8}
E(a_0)<0 \with E(a)\eqdefa \int_0^\infty\Bigl(\f12(\p_ya(y))^2-\f14a^3(y)\Bigr)\,dy<0.
\eeq
Then any smooth solution of \eqref{S1eq1} does not exist globally in time.

For small initial data $u_0(x,y)=\eta\phi(x,y),$  we have $a_0(y)=-\eta\p_x\phi(0,y)$ and
\beno
E(a_0)=\f{\eta^2}2\int_0^\infty \left(\p_x\p_y\phi(0,y)\right)^2\,dy-\f{\eta^3}4\int_0^\infty \left(\p_x\phi(0,y)\right)^3\,dy,
\eeno
which can not satisfy $E(a_0)<0$ for $\eta$ sufficiently small except that $\p_x\p_y\phi(0,y)=0.$ However, in the later  case,
due to the fact that the solution decays to zero as $y$ approaching to $+\infty,$  we have $a_0(y)=\p_x\phi(0,y)=0,$
which implies $E(a_0)=0$ so that \eqref{S1eq8} can not be satisfied in both cases.

\item[(3)] We also remark  that the exponential weight that appears in the norm of \eqref{S1eq6} excludes
 the possibility of  taking initial data of \eqref{S1eq5} which is slowly varying  in the normal variable.
Indeed we consider an initial data of the form $u_0^\e(x,y)=\eta \phi\bigl(x,\e{y}\bigr)$ with  $\eta, \e$ being sufficiently small
 such hat
\beno
E(a_0)=\f{\e\eta^2}{2}\int_0^\infty \left(\p_x\p_y\phi(0,y)\right)^2\,dy-\f{\eta^3 }{4\e}\int_0^\infty \left(\p_x\phi(0,y)\right)^3\,dy<0.
\eeno
Then it is easy to check that $u_0^\e$  defined above can not verify our smallness condition  \eqref{S1eq6}.
\end{itemize}

\end{rmk}

\begin{rmk}

\begin{itemize}

\item[(1)]
 The idea of  closing the analytic energy
estimate, \eqref{S1eq7}, for  solutions of \eqref{S1eq5}  goes back to
\cite{Ch04} where Chemin introduced a tool to  make analytical type estimates and
controlling the size of the analytic radius simultaneously. It was used in the context of
anisotropic Navier-Stokes system \cite{CGP} ( see also \cite{mz1, mz2}), which implies  the global
well-posedness of three dimensional Navier-Stokes system with a
class of ``ill prepared data", which is slowly varying in the
vertical variable, namely of the form $\e x_3,$  and the $B^{-1}_{\infty,\infty}(\R^3)$
norm of which blow up as the small parameter goes to zero.

\item[(2)]
We mention that in our previous paper with Zhang in \cite{PZZ2}, we used
the weighted analytic norm of $\p_yu$  to control the analytic band of the solutions, which seems more obvious than the
weighted
analytic norm of $g,$ which is defined by \eqref{S7eq4}.
Since in \cite{PZZ2}, we worked on Prandt type system in a strip with homogenous boundary condition
so that we can use the classical Poincar\'e inequality to derive the exponential decay estimates for the solutions. Therefore
we have a global control for the analytic band. Here
in the upper space, by using another type of Poincar\'e inequality,  \eqref{S2eq1}, which yields decay of a sort of
weighted analytic norm to $\p_yu$ like $\w{t}^{-\f54}$ as the time $t$ going to $\infty.$
 Yet this estimate  can not guarantee the quantity: $\int_0^\infty \w{t}^{\f14}\|e^\Psi\p_yu_\Phi(t)\|_{\cB^{\f12,0}}\,dt,$
 to be finite, which will be crucial to globally control the analytic band of the solutions to \eqref{S1eq5}.
\end{itemize}
\end{rmk}

\smallskip

 Let us end this introduction by the notations that we shall use
in this context.\vspace{0.2cm}

For~$a\lesssim b$, we mean that there is a uniform constant $C,$
which may be different on different lines, such that $a\leq Cb$. $(a\ |\
b)_{L^2_+}\eqdefa\int_{\R^2_+}a(x,y) {b}(x,y)\,dx\,dy$ (resp. $(a\ |\
b)_{L^2_{\rm v}}\eqdefa\int_{\R_+}a(y) {b}(y)\,dy$)  stands for
the $L^2$ inner product of $a,b$ on $\R^2_+$ (resp. $\R_+$) and
$L^p_+=L^p(\R^2_+)$  with $\R^2_+\eqdefa\R\times\R_+.$ For $X$ a Banach space
and $I$ an interval of $\R,$ we denote by $L^q(I;\,X)$ the set of
measurable functions on $I$ with values in $X,$ such that
$t\longmapsto\|f(t)\|_{X}$ belongs to $L^q(I).$  In particular,  we denote by
$L^p_T(L^q_{\rm h}(L^r_{\rm v}))$ the space $L^p([0,T];
L^q(\R_{x_{}};L^r(\R_{y}^+))).$  Finally,
 $
\left(d_{k}\right)_{k\in\Z}$  designates  a nonnegative generic element in the sphere
of $\ell^1(\Z)$
 so that  $\sum_{k\in\Z}d_k=1.$

\smallskip

\renewcommand{\theequation}{\thesection.\arabic{equation}}
\setcounter{equation}{0}

\section{Littlewood-Paley theory and functional framework}\label{sect2}

In the rest of this paper, we shall frequently use Littlewood-Paley decomposition in the horizontal variable, $x.$
For the convenience of the readers,  we shall collect some basic facts on anisotropic  Littlewood-Paley theory in this section.
Let us first recall from
\cite{BCD} that \beq
\begin{split}
&\Delta_k^{\rm h}a=\cF^{-1}(\varphi(2^{-k}|\xi|)\widehat{a}),\qquad
S^{\rm h}_ka=\cF^{-1}(\chi(2^{-k}|\xi|)\widehat{a}),
\end{split} \label{1.3a}\eeq where and in all that follows, $\cF
a$ and $\widehat{a}$ always denote the partial  Fourier transform of
the distribution $a$ with respect to $x$ variable,  that is, $
\widehat{a}(\xi,y)=\cF_{x\to\xi}(a)(\xi,y),$
  and $\chi(\tau),$ ~$\varphi(\tau)$ are
smooth functions such that
 \beno
&&\Supp \varphi \subset \Bigl\{\tau \in \R\,/\  \ \frac34 \leq
|\tau| \leq \frac83 \Bigr\}\andf \  \ \forall
 \tau>0\,,\ \sum_{k\in\Z}\varphi(2^{-k}\tau)=1,\\
&&\Supp \chi \subset \Bigl\{\tau \in \R\,/\  \ \ |\tau|  \leq
\frac43 \Bigr\}\quad \ \ \ \andf \  \ \, \chi(\tau)+ \sum_{k\geq
0}\varphi(2^{-k}\tau)=1.
 \eeno


\begin{Def}\label{def1.2}
{\sl  Let~$s$ in~$\R$. For~$u$ in~${\mathcal S}_h'(\R^2_+),$ which
means that $u$ is in~$\cS'(\R^2_+)$ and
satisfies~$\lim_{k\to-\infty}\|S_k^{\rm h}u\|_{L^\infty}=0,$ we set
$$
\|u\|_{\cB^{s,0}}\eqdefa\big\|\big(2^{ks}\|\Delta_k^{\rm h}
u\|_{L^{2}_+}\big)_{k\in\Z}\bigr\|_{\ell ^{1}(\Z)}.
$$
\begin{itemize}

\item
For $s\leq \frac{1}{2}$, we define $ \cB^{s,0}(\R^2_+)\eqdefa
\big\{u\in{\mathcal S}_h'(\R^2_+)\;\big|\; \|
u\|_{\cB^{s,0}}<\infty\big\}.$

\item
If $\ell$ is  a positive integer and if~$\ell-\frac{1}{2}< s\leq
\ell+\frac{1}{2}$, then we define~$ \cB^{s,0}(\R^2_+)$  as the subset
of distributions $u$ in~${\mathcal S}_h'(\R^2_+)$ such that
$\p_x^\ell u$ belongs to~$ \cB^{s-\ell,0}(\R^2_+).$
\end{itemize}
}
\end{Def}

In  order to obtain a better description of the regularizing effect
of the transport-diffusion equation, we need to use Chemin-Lerner
type spaces $\widetilde{L}^{p}_T(\cB^{s,0}(\R^2_+))$.
\begin{Def}\label{def2.2}
{\sl Let $p\in[1,\,+\infty]$ and $T_0,T\in[0,\,+\infty]$. We define
$\widetilde{L}^{p}(T_0,T; \cB^{s,0}(\R^2_+))$ as the completion of
$C([T_0,T]; \,\cS(\R^2_+))$ by the norm
$$
\|a\|_{\widetilde{L}^{p}(T_0,T; \cB^{s,0})} \eqdefa \sum_{k\in\Z}2^{ks}
\Big(\int_{T_0}^T\|\Delta_k^{\rm h}\,a(t) \|_{L^2_+}^{p}\,
dt\Big)^{\frac{1}{p}}
$$
with the usual change if $p=\infty.$ In particular, when $T_0=0,$ we shall denote $\|a\|_{\widetilde{L}_T^{p}(\cB^{s,0})}
\eqdefa \|a\|_{\widetilde{L}^{p}(0,T; \cB^{s,0})}$ for simplicity.}
\end{Def}

In order to overcome the difficulty that one can not use Gronwall's
type argument in the framework of Chemin-Lerner space
$\wt{L}^2_T(\cB^{s,0}),$ we also need to use the time-weighted
Chemin-Lerner type norm, which was introduced by the authors in
\cite{PZ1}.

\begin{Def}\label{def1.1} {\sl Let $f(t)\in L^1_{\mbox{loc}}(\R_+)$
be a nonnegative function and $t_0, t\in [0,\infty].$ We define \beq \label{1.4}
\|a\|_{\wt{L}^p_{t_0,t;f}(\cB^{s,0})}\eqdefa
\sum_{k\in\Z}2^{ks}\Bigl(\int_{t_0}^t f(t')\|\D_k^{\rm
h}a(t')\|_{L^2_+}^p\,dt'\Bigr)^{\f1p}. \eeq When $t_0=0,$ we simplify the notation $\|a\|_{\wt{L}^p_{0,t:f}(\cB^{s,0})}$ as $\|a\|_{\wt{L}^p_{t,f}(\cB^{s,0})}.$}
\end{Def}

 \medbreak
We also recall the following anisotropic
Bernstein  lemma from \cite{CZ1, Pa02}:

\begin{lem} \label{lem:Bern}
 {\sl Let $\cB_{\rm h}$ be a ball
of~$\R_{\rm h}$, and~$\cC_{\rm h}$  a ring of~$\R_{\rm
h}$; let~$1\leq p_2\leq p_1\leq \infty$ and ~$1\leq q\leq \infty.$
Then there holds:

\smallbreak\noindent If the support of~$\wh a$ is included
in~$2^k\cB_{\rm h}$, then
\[
\|\partial_{x}^\ell a\|_{L^{p_1}_{\rm h}(L^{q}_{\rm v})} \lesssim
2^{k\left(\ell+\f 1 {p_2}-\f 1 {p_1}\right)}
\|a\|_{L^{p_2}_{\rm h}(L^{q}_{\rm v})}.
\]

\smallbreak\noindent If the support of~$\wh a$ is included
in~$2^k\cC_{\rm h}$, then
\[
\|a\|_{L^{p_1}_{\rm h}(L^{q}_{\rm v})} \lesssim
2^{-k\ell} \|\partial_{x}^\ell a\|_{L^{p_1}_{\rm
h}(L^{q}_{\rm v})}.
\]
}
\end{lem}

Finally to deal with the estimate concerning the product of two distributions,
we shall frequently use  Bony's decomposition (see \cite{Bo}) in
the horizontal variable:
 \ben\label{Bony} fg=T^{\rm h}_fg+T^{\rm
h}_{g}f+R^{\rm h}(f,g), \een where \beno
\begin{split}
 T^{\rm h}_fg\eqdefa\sum_kS^{\rm
h}_{k-1}f\Delta_k^{\rm h}g,\quad R^{\rm
h}(f,g)\eqdefa&\sum_k\widetilde{\Delta}_k^{\rm h}f\Delta_{k}^{\rm h}g \with
\widetilde{\Delta}_k^{\rm h}f\eqdefa
\displaystyle\sum_{|k-k'|\le 1}\Delta_{k'}^{\rm h}f. \end{split} \eeno

\renewcommand{\theequation}{\thesection.\arabic{equation}}
\setcounter{equation}{0}

\section{Sketch of the proof to Theorem \ref{th1.1}}

We point out that a key ingredient used in the proof of Theorem \ref{th1.1} is the following Poincar\'e type inequality,
which is a special case of Treves inequality that can be found in \cite{Hor83} (see also  Lemma 3.3 of \cite{IV16}).

\begin{lem}\label{lem2.1}
{\sl Let $\Psi(t,y)\eqdefa \frac{y^2}{8\w{t}}$ and $d$ be a nonnegative integer. Let $u$ be a smooth enough function on $\R^d\times\R_+$
which decays to zero sufficiently fast as $y$ approaching to $+\infty.$  Then one has
\beq \label{S2eq1}
\int_{\R^d\times\R_+} |\partial_y u(X,y)|^2 e^{2\Psi}\,dX\,dy \geq \frac1{2\w{t}}\int_{\R^d\times\R_+} |u(X,y)|^2 e^{2\Psi}\,dX\,dy.\eeq
}
\end{lem}

\begin{proof} We remark that compared with Lemma 3.3 of \cite{IV16}, here we do not need any boundary condition for
$u$ on the boundary $y=0.$ For completeness, we outline its proof here.  As a matter of fact, for any fixed $X\in\R^d,$
we first get, by using integration by parts, that
\beno
\begin{split}
\int_{R_+} u^2(X,y) e^{\frac{y^2}{4\w{t}}} \,dy=&\int_{R_+} (\partial_y y) u^2(X,y) e^{\frac{y^2}{4\w{t}}}\,dy\\
=&-2\int_{R_+} y u(X,y) \partial_y u(X,y) e^{\frac{y^2}{4\w{t}}}\,dy-\frac{1}{2\w{t}}\int_{R_+} y^2 u^2(X,y) e^{\frac{y^2}{4\w{t}}}\,dy.
\end{split}
\eeno
By integrating the above inequality over $\R^d$ with respect to the $X$ variables, we find
\beno
\begin{split}
\int_{\R^d\times\R_+} u^2& e^{\frac{y^2}{4\w{t}}} \,dX\,dy+\frac{1}{2\w{t}}\int_{\R^d\times\R_+} y^2 u^2 e^{\frac{y^2}{4\w{t}}}\,dX\,dy
=-2\int_{\R^d\times\R_+} y u \partial_y u e^{\frac{y^2}{4\w{t}}}\,dX\,dy\\
\leq& 2\bigg(\frac{1}{2\w{t}}\int_{\R^d\times\R_+} y^2 u^2 e^{\frac{y^2}{4\w{t}}}\,dX\,dy\bigg)^{1/2}\bigg( 2\w{t}\int_{\R^d\times\R_+} (\partial_y u)^2 e^{\frac{y^2}{4\w{t}}}\,dX\, dy\bigg)^{1/2}\\
\leq & \frac{1}{2\w{t}}\int_{\R^d\times\R_+} y^2 u^2 e^{\frac{y^2}{4\w{t}}}\,dX\,dy+2\w{t}\int_{\R^d\times\R_+} (\partial_y u)^2 e^{\frac{y^2}{4\w{t}}}\,dX\,dy.
\end{split}
\eeno
This leads to \eqref{S2eq1}.
\end{proof}

By virtue of  Lemma \ref{lem2.1},
we get, by using  a standard argument of energy estimate to the system \eqref{S1eq4}, that
\beq \label{Soeq1}
\|e^{\f{y^2}{8\w{t}}}\p_yu^s(t)\|_{L^2_{\rm v}}\leq C\w{t}^{-\f34}.
\eeq
We remark that intuitively  the quantity $\|e^\Psi\p_yu^s(t)\|_{L^2_{\rm v}}$ is a natural part to control the
time evolution of  the analytical radius to the analytic solutions of \eqref{S1eq1}. Yet it is obvious that \eqref{Soeq1}
 is not enough to guarantee that the quantity $\int_0^\infty\w{t}^\f14\|e^{\f{y^2}{8\w{t}}}\p_yu^s(t)\|_{L^2_{\rm v}}\,dt$ is finite,
 which will be required to go through our process below.

  To overcome
the above difficulty,  we are going to construct a special solution of \eqref{S1eq4} via its primitive function, that is,
 $u^s(t,y)=\p_y\psi^s(t,y).$  And we define $\psi^s$ through
\begin{equation}\label{S2eq2}
\left\{
\begin{array}{ll}
\p_t\psi^s-\p_y^2\psi^s=\e M(t,y), \quad (t,y)\in\R_+\times\R_+, \\
\psi^s|_{y=0}=0 \andf \lim_{y\to +\infty} \psi^s=0,\\
\psi^s|_{t=0}=0,
\end{array}
\right.
\end{equation}
where \beq \label{S2eq3}
M(t,y)\eqdefa -\int_y^\infty(1-\chi(y'))\,dy'f'(t)+ f(t)\chi'(y),
\eeq so that $m$ in \eqref{S1eq3} equals to $\p_yM.$ We observe that
$\int_y^\infty(1-\chi(y'))\,dy'=0$ for $y\geq 2,$  that is, $M(t,y)$ is supported on the interval $[0,2]$ with
respect to $y$ variable. It is crucial to observe that the quantity
\beq \label{S2eq4} G^s\eqdefa u^s+\f{y}{2\w{t}}\psi^s \eeq decays faster than $u^s,$ which is inspired
the definition of the function $\frak{g}$ in \cite{IV16}.

Indeed we first observe from \eqref{S2eq2} that
\beno
\p_t\Bigl(\f{y}{2\w{t}}\psi^s\Bigr)-\p_y^2\Bigl(\f{y}{2\w{t}}\psi^s\Bigr)+\f1{\w{t}}G^s=\e\f{y}{2\w{t}}M,
\eeno
from which, \eqref{S1eq4} and \eqref{S2eq4}, we find
\beq \label{S2eq7}
\left\{
\begin{array}{ll}
\p_tG^s-\p_y^2G^s+\w{t}^{-1}G^s=\e H \with H\eqdefa  m+\f{y}{2\w{t}}M,\\
G^s|_{y=0}=0 \andf \lim_{y\to +\infty} G^s(t,y)=0,\\
G^s|_{t=0}=0,
\end{array}
\right.
\eeq
With  $G^s$ being determined by \eqref{S2eq7},  by virtue of \eqref{S2eq4} and $\psi^s|_{y=0}=0,$ we obtain
\beq \label{S2eq7up}
\psi^s(t,y)= e^{-\f{y^2}{4\w{t}}}\int_0^ye^{-\f{(y')^2}{4\w{t}}}G^s(t,y')\,dy' \andf u^s(t,y)\eqdefa \p_y\psi^s(t,y).
\eeq
We observe that $u^s$ defined above satisfies the boundary condition $u^s(t,0)=0$ although the boundary condition
in \eqref{S2eq2} does not match with that in \eqref{S1eq4}.

As we already mentioned in \eqref{Soeq1}, a similar decay estimate for the weighted analytical norm to the solutions
of \eqref{S1eq5} can also be derived, which will not be enough to go through our process below.
 Inspired by the function $\frak{g}\eqdefa \p_yu+\f{y}{2\w{t}}u,$
which was introduced by Ignatova and Vicol in \cite{IV16}, here we introduce the function $G$ and $g$ in \eqref{S7eq4}.
It is a crucial observation here that the weighted analytic norm of $g$ can control the evolution of the analytic norm
to the solutions of \eqref{S1eq5}.

Next as in \cite{Ch04, CGP, mz1, mz2, ZHZ}, for any locally bounded function
$\Phi$ on $\R^+\times \R$, we define \beq\label{eq2.4}
u_\Phi(t,x,y)=\cF_{\xi\to
x}^{-1}\bigl(e^{\Phi(t,\xi)}\widehat{u}(t,\xi,y)\bigr). \eeq
Let $G$  and $G^s$ be determined respectively by \eqref{S7eq4} and \eqref{S2eq4},
we
introduce a key quantity $\theta(t)$ to describe the evolution of
the analytic band to the solutions of \eqref{S1eq5}:
\begin{equation}\label{1.9}
 \quad\left\{\begin{array}{l}
\displaystyle \dot{\tht}(t)=\w{t}^\f14\bigl(\|e^\Psi\p_yG^s(t)\|_{L^2_{\rm v}}+\e f(t)\|e^\Psi\chi'\|_{L^2_{\rm v}}+\|e^\Psi \p_yG_\Phi(t)\|_{\cB^{\f{1}2,0}}\bigr),\\
\displaystyle \tht|_{t=0}=0.
\end{array}\right.
\end{equation}
Here $\w{t}\eqdefa 1+t,$  the phase function $\Phi$ is defined by
\beq\label{eq2.6} \Phi(t,\xi)\eqdefa (\de-\lam \tht(t))|\xi|, \eeq
and the weighted function $\Psi(t,y)$ is determined by \beq
\label{eq2.7} \Psi(t,y)\eqdefa \f{y^2}{8\w{t}}, \eeq which
satisfies \beq\label{eq2.8} \p_t\Psi(t,y)+2(\p_y\Psi(t,y))^2= 0.
\eeq

We present now a more precise statement of our result in this paper.

\begin{thm}\label{th1.3}
{\sl Let $\Phi$ and $\Psi$ be defined respectively  by \eqref{eq2.6} and \eqref{eq2.7}. Then under the assumptions of Theorem \ref{th1.1},  there exist positive constants $c_0,$ $ \e_0$  and  $\lam$ so that for
 $u^s$ determined by \eqref{S2eq7up} and $\varepsilon\leq \e_0,$
 the system (\ref{S1eq5}) has a  unique global solution $u$ which satisfies $\sup_{t\in [0,\infty)}\theta(t)\leq \f\delta{2\lam},$ and
 \beq\label{eq7}
\bigl\| e^{\Psi} u_\Phi \bigr\|_{\wt{L}^\infty(\R_+;\cB^{\f12,0})}+\bigl\|e^{\Psi} \p_y u_\Phi\bigr\|_{\wt{L}^2(\R_+;\cB^{\f12,0})}\leq C\bigl\|e^{\f{y^2}{8}} e^{\delta |D_x|}u_0\bigr\|_{\cB^{\f12,0}}.
\eeq
Moreover, for $G$  given by \eqref{S7eq4},  there exists a positive constant $C$ so that
for any $t>0$ and $\ga\in(0,1),$ there hold
\beq \label{eq10}
\begin{split}
&\bigl\|\w{t'}^{\f34} e^{\Psi} u_\Phi\bigr\|_{{L}^\infty(\R^+;\cB^{\f12,0})}+\bigl\|\w{t'}^{\f34} e^{\Psi} \p_y u_\Phi\bigr\|_{\wt{L}^2(t/2,t;\cB^{\f12,0})}\leq C\|e^{\f{y^2}{8}}e^{\delta|D_x|}(\varphi_0,u_0)\|_{\cB^{\f12,0}},\\
&\bigl\|\w{t'}^{\f54} e^{\Psi} G_\Phi\bigr\|_{\wt{L}^\infty(\R^+;\cB^{\f12,0})}+\bigl\|\w{t'}^{\f54} e^{\Psi} \p_y G_\Phi\bigr\|_{\wt{L}^2(t/2,t;\cB^{\f12,0})}\leq C\|e^{\f{y^2}{8}}e^{\delta|D_x|}G_0\|_{\cB^{\f12,0}},\\
&\bigl\|\w{t'}^{\f54} e^{\ga\Psi} u_\Phi\bigr\|_{\wt{L}^\infty(\R^+;\cB^{\f12,0})}+\bigl\|\w{t'}^{\f54} e^{\ga\Psi} \p_y u_\Phi\bigr\|_{\wt{L}^2(t/2,t;\cB^{\f12,0})}\leq C\|e^{\f{y^2}{8}}e^{\delta|D_x|}G_0\|_{\cB^{\f12,0}}.
\end{split}
\eeq
}
\end{thm}

We remark that one of the crucial step to prove Theorem \ref{th1.3} is to control
the time evolution of  $\theta(t),$ which basically determines the analytical radius
 of the solutions to \eqref{S1eq5}. \smallskip

  Let us now sketch the structure of this paper below. \smallskip

 In Section \ref{sect1a}, we  shall prove the following proposition concerning the large time
 decay estimate of  $\|e^\Psi G^s(t)\|_{L^2_{\rm v}},$
 which in particular guarantees  that

\begin{prop}\label{prop2.1}
{\sl Let  $f(t)\in H^1(\R_+)$  and satisfy \eqref{S2eq12}. Then for  $G^s$ being determined by  \eqref{S2eq7},
 one has
\beq \label{S2eq9}
\int_{0}^\infty \w{t}^{\f14}\|e^\Psi\p_y G^s(t)\|_{L^2_{\rm v}}\,dt\leq  C\cC_{f}\e,
\eeq for the constant $\cC_{f}$ given by \eqref{S2eq12}.}\end{prop}

In what follows, we shall always assume that $t<T^\ast$ with
$T^\ast$ being determined by \beq\label{1.8a} T^\ast\eqdefa
\sup\bigl\{\ t>0,\ \ \tht(t) <\de/\lam\bigr\}. \eeq So that by
virtue of \eqref{eq2.6}, for any $t<T^\ast,$ there holds the
following convex inequality \beq\label{1.8bb} \Phi(t,\xi)\leq
\Phi(t,\xi-\eta)+\Phi(t,\eta)\quad\mbox{for}\quad \forall\
\xi,\eta\in \R. \eeq

In Section \ref{Sect4}, we shall deal with the {\it a priori} decay estimates for the analytic solutions of \eqref{S1eq2}.

\begin{prop}\label{S4prop2}
{\sl Let $\vf$ be a smooth enough solution of \eqref{S1eq2}. Then there exists a large enough constant
 $\lam$ so that for any nonnegative and non-decreasing function $h\in C^1(\R_+)$ and any $t_0\in  [0,t]$ with $t<T^\ast,$ one has
 \beq \label{S4eq11ag}
\|\w{t'}^{\f14} e^\Psi
\vf_\Phi\|_{\wt{L}^\infty_t(\cB^{\f{1}2,0})}\leq
C\|e^{\f{y^2}8} e^{\delta|D_x|}\vf_0\|_{\cB^{\f{1}2,0}}, \eeq
and
\beq \label{S4eq16}
\begin{split}
\|\hbar^{\f12}e^{\Psi}\vf_\Phi\|_{\wt{L}^\infty(t_0,t;\cB^{\f12,0})}
&+
\|\hbar^{\f12}e^{\Psi}\p_y\vf_\Phi\|_{\wt{L}^2(t_0,t;\cB^{\f12,0})}\\
&\leq \|\hbar^{\f12}e^{\Psi}\vf_\Phi(t_0)\|_{\cB^{\f12,0}}+\bigl\|\sqrt{\hbar'}e^{\Psi}\vf_\Phi\|_{\wt{L}^2(t_0,t;\cB^{\f12,0})}.
\end{split}
\eeq
}
\end{prop}

Section \ref{Sect5} is devoted to the {\it a priori} decay estimates for the analytic solutions of \eqref{S1eq5}:

\begin{prop}\label{S5prop2}
{\sl Let $u$ be a smooth enough solution of \eqref{S1eq5}. Then there exists a large enough constant
 $\lam$ so that for any $t<T^\ast,$ we have
\beq \label{S4eq12}
\begin{split} \|\w{t'}^{\f34}e^\Psi
u_\Phi\|_{{L}^\infty_t(\cB^{\f{1}2,0})}+&\|\w{t'}^{\f34} e^\Psi
\p_y u_\Phi\|_{\wt{L}^2(t/2,t;\cB^{\f12,0})}\leq
C\|e^{\f{y^2}8} e^{\delta|D_x|}\left(\vf_0,u_0\right)\|_{\cB^{\f{1}2,0}}.\end{split} \eeq
}
\end{prop}

In Section \ref{Sect7}, we shall deal with the {\it a priori} decay estimates of $G,$ which will be the most crucial ingredient
used in the proof of Theorem \ref{th1.3}.

\begin{prop}\label{S7prop2}
{\sl Let $G$ be determined by \eqref{S7eq4}. Then there exists a large enough constant
 $\lam$ so that for any $t<T^\ast,$ we have
\beq \label{S7eq24}
\begin{split} \|\w{t'}^{\f54}e^\Psi
G_\Phi\|_{\wt{L}^\infty_t(\cB^{\f{1}2,0})}+&\|\w{t'}^{\f54} e^\Psi
\p_y G_\Phi\|_{\wt{L}^2(t/2,t;\cB^{\f12,0})}\\
+&\int_0^t\w{t'}^{\f14}\|e^\Psi
\p_y G_\Phi(t')\|_{\cB^{\f12,0}}\,dt'
\leq
C\|e^{\f{y^2}8} e^{\delta|D_x|}G_0\|_{\cB^{\f{1}2,0}}.\end{split} \eeq
}
\end{prop}

With the above propositions,  we still need the follow lemma concerning the relations between the function $G$ given by \eqref{S7eq4}
and the solutions of \eqref{S1eq5} and \eqref{S1eq2},  which will
be also frequently used in the subsequent sections.

\begin{lem}\label{S0lem1}
{\sl Let $G$ and $\Psi$ be defined respectively  by \eqref{S7eq4} and \eqref{eq2.7}.
Let $\vf$ and $u$ be  smooth enough solution of \eqref{S1eq2} and \eqref{S1eq5} respectively on $[0,T].$
 Then, for any $
\ga\in (0,1)$ and $t\leq T,$ one has
\ben
&&
\bigl\|e^{\ga\Psi}\D_k^\h u_\Phi(t)\bigr\|_{L^2_+}\lesssim \bigl\|e^{\Psi}\D_k^\h G_\Phi(t)\bigr\|_{L^2_+}; \label{S2eq21}\\
&&
\bigl\|e^{\ga\Psi}\D_k^\h\p_yu_\Phi(t)\bigr\|_{L^2_+}\lesssim \bigl\|e^{\Psi}\D_k^\h\p_yG_\Phi(t)\bigr\|_{L^2_+}; \label{S2eq20}\\
&&
\w{t}^{-1}\|e^{\ga\Psi}\D_k^\h\p_y(y\vf)_\Phi(t)\|_{L^2_+}+\w{t}^{-\f34}\|e^{\ga\Psi}\D_k^\h\p_y(y\vf)_\Phi(t)\|_{L^\infty_{\rm v}(L^2_\h)}
\lesssim \|e^\Psi\D_k^\h \p_yG_\Phi\|_{L^2_+}. \label{S7eq21} \een}
\end{lem}

Let us postpone the proof of this lemma till the end of this section.

\smallskip

We are now in a position to complete the proof of Theorem
\ref{th1.3}.\smallskip

\begin{proof}[Proof of Theorem \ref{th1.3}]
The general strategy to prove the existence result for a nonlinear
partial differential equation is first to construct  appropriate
approximate solutions, then perform uniform estimates for such
approximate solution sequence, and finally pass to the limit in the
approximate problem. For simplicity, here we only present the {\it a
priori} estimates for smooth enough solutions of \eqref{S1eq5} in the
analytical framework.

Indeed let $u$ and $\vf$ be smooth enough solutions of \eqref{S1eq5} and \eqref{S1eq2} respectively on $[0,T^\star),$
where $T^\star$ is the maximal time of existence of the solutions.
Let $G$ be defined by \eqref{S7eq4}. For any
$t<T^\ast$ (of course here $T^\ast\leq T^\star$) with $T^\ast$ being defined by \eqref{1.8a},  we deduce from \eqref{1.9} that
\beno
\tht(t)\leq  \int_0^t\w{t'}^\f14\bigl(\|e^\Psi\p_yG^s(t')\|_{L^2_{\rm v}}+\|e^\Psi
\p_yG_\Phi(t')\|_{\cB^{\f{1}2,0}}\bigr)\,dt'+\e\int_0^t\|e^{\Psi(t')}\chi'\|_{L^2_{\rm v}}
\w{t'}^{\f14}f(t')\,dt'.
\eeno
Notice that $\Supp\chi'\subset [1,2],$ one has
\beno
\|e^{\Psi(t')}\chi'\|_{L^2_{\rm v}}\leq e^{\f1{2\w{t'}}}\|\chi'\|_{L^2_{\rm v}}\leq e^{\f12}\|\chi'\|_{L^2_{\rm v}},
\eeno
from which, Proposition \ref{prop2.1} and Proposition \ref{S7prop2}, we infer
\beq \label{S5eq9}
\tht(t)\leq C\bigl(\|e^{\f{y^2}8} e^{\delta|D_x|}G_0\|_{\cB^{\f{1}2,0}}+\e \cC_{f}\bigr)\quad\mbox{for}\quad t<T^\ast,
\eeq
where the constant $ \cC_{f}$ is  determined   by  \eqref{S2eq12}.

In particular, if we take $c_0$ in \eqref{S1eq6} and $\e_0$ so small that
\beq \label{S5eq10}
C\left(c_0+\e_0 \cC_{f}\right)\leq \f{\delta}{2\lambda}.
\eeq
Then we deduce from \eqref{S5eq9} that
\beno
\sup_{t\in [0,T^\ast)}\tht(t)\leq \frac{\delta}{2\lambda}\quad\mbox{for}
\quad \e\leq \e_0.
\eeno
So that in view of \eqref{1.8a}, we get by a continuous argument that
$T^\ast=\infty.$   And  Propositions \ref{S5prop2} and \ref{S7prop2} ensure the first two inequalities
 of \eqref{eq10}. Moreover, \eqref{1.19} holds for $t=\infty,$  which implies
\eqref{eq7}. Finally Proposition  \ref{S7prop2} and Lemma \ref{S0lem1} ensures  the last inequality of \eqref{eq10}.
This completes the existence part of Theorem
\ref{th1.3}. The uniqueness part follows from Theorem 1.1 of \cite{ZHZ}. This concludes the proof of Theorem
\ref{th1.3}.
\end{proof}

Let us end this section with the proof of Lemma \ref{S0lem1}.

\begin{proof}[Poor of Lemma \ref{S0lem1}] As a matter of fact, due to $\p_xu+\p_yv=0$ and
$v(t,x,0)=0,$ we find
\beno
\p_x\int_0^\infty u(t,x,y)\,dy=-\int_0^\infty \p_yv(t,x,y)\,dy=v(t,x,0)=0,
\eeno
which implies $\int_0^\infty u(t,x,y)\,dy=C(t).$
 Yet since $u$ decays to zero as $|x|$ tends to $\infty,$ we have $C(t)=0,$
that is
\beq \label{S2eq15}
\int_0^\infty u(t,x,y)\,dy=0.
\eeq
Due to $u=\p_y\vf,$ we deduce that
\beq \label{S2eq16}
\vf(t,x,0)=-\int_0^\infty u(t,x,y)\,dy=0.
\eeq
Thanks to \eqref{S2eq16}, we deduce from \eqref{S7eq4} and $u=\p_y\vf$ that
\beq \label{S2eq17}
\vf(t,x,y)=e^{-\f{y^2}{4\w{t}}}\int_0^ye^{\f{(y')^2}{4\w{t}}}G(t,x,y')\,dy',
\eeq
which implies
\beq \label{S2eq18}
u=\p_y\vf=-\f{y}{2\w{t}}e^{-\f{y^2}{4\w{t}}}\int_0^ye^{\f{(y')^2}{4\w{t}}}G(t,x,y')\,dy'+G,
\eeq
and
\beq\label{S2eq19}
\begin{split}
\p_yu=\p_y^2\vf=&-\f{y}{2\w{t}}G+\p_yG(t,y)\\
&+\Bigl(-\f{1}{2\w{t}}+\f{y^2}{4\w{t}^2}\Bigr)e^{-\f{y^2}{4\w{t}}}\int_0^ye^{\f{(y')^2}{4\w{t}}}G(t,x,y')\,dy'.
\end{split}
\eeq

In view of \eqref{eq2.7}, \eqref{S2eq18} and
\beq
\label{S2eq19p} \sup_{y\in [0,\infty)}\bigl(e^{-y^2}\int_0^y e^{z^2}\,dz\bigr)<\infty,\eeq we infer that
\beno
\begin{split}
\bigl\|e^{\ga\Psi}\D_k^\h u_\Phi(t)\bigr\|_{L^2_+}\lesssim &\bigl\|e^{\Psi}\D_k^\h G_\Phi(t)\bigr\|_{L^2_+}\\
&+\w{t}^{-1}\bigl\|ye^{(\ga-2)\Psi}\Bigl(\int_0^ye^{2\Psi}\,dy'\Bigr)^{\f12}
\Bigl(\int_0^\infty |e^\Psi\D_k^\h G|^2\,dy\Bigr)^{\f12}\bigr\|_{L^2_+}\\
\lesssim & \bigl\|e^{\Psi}\D_k^\h G_\Phi(t)\bigr\|_{L^2_+}+\w{t}^{-\f34}\bigl\|ye^{-(1-\ga)\Psi}
\Bigl(\int_0^\infty |e^\Psi\D_k^\h G|^2\,dy\Bigr)^{\f12}\bigr\|_{L^2_+},
\end{split}
\eeno
from which and $\ga\in (0,1),$ we deduce \eqref{S2eq21}.

Whereas due to  $\lim_{y\to \infty}G(t,x,y)=0,$ we write $G=-\int_y^\infty \p_yG\,dy',$ and hence, for $\ga\in (0,1),$ we infer
\beno
\begin{split}
\|\D_k^\h G_\Phi(t,\cdot,y)\|_{L^2_\h}\leq& \Bigl(\int_y^\infty e^{-2\Psi}\,dy'\Bigr)^{\f12}\Bigl(\int_0^\infty |e^{\Psi} \D_k^\h\p_yG_\Phi|^2\,dy\Bigr)^{\f12}\\
\lesssim &\w{t}^{\f14} e^{-\f{1+\ga}2\Psi}\Bigl(\int_0^\infty |e^{\Psi} \D_k^\h\p_yG_\Phi|^2\,dy\Bigr)^{\f12},
\end{split}
\eeno
from which and \eqref{S2eq19}, we infer
\beq \label{S2eq20op}
\begin{split}
\bigl\|e^{\ga\Psi}\D_k^\h\p_yu_\Phi(t)\bigr\|_{L^2_+}
\lesssim &  \w{t}^{-\f34}\bigl\|e^{(\ga-2)\Psi}\int_0^ye^{\f{3-\ga}2\Psi}\,dy'\Bigl(\int_0^\infty |e^{\Psi} \D_k^\h\p_yG_\Phi|^2\,dy\Bigr)^{\f12}\bigr\|_{L^2_+}\\
&+\w{t}^{-\f74}\bigl\|y^2e^{(\ga-2)\Psi}\int_0^ye^{\f{3-\ga}2\psi}\,dy'\Bigl(\int_0^\infty |e^{\Psi} \D_k^\h\p_yG_\Phi|^2\,dy\Bigr)^{\f12}\bigr\|_{L^2_+}\\
&+\w{t}^{-\f34}\bigl\|ye^{\f{\ga-1}2\Psi}\Bigl(\int_0^\infty |e^{\Psi} \D_k^\h\p_yG_\Phi|^2\,dy\Bigr)^{\f12}\bigr\|_{L^2}\\
&\qquad\qquad\qquad\qquad\qquad\qquad\qquad\qquad+\bigl\|e^{\Psi}\D_k^\h\p_yG_\Phi(t)\bigr\|_{L^2_+}.
\end{split}
\eeq
\eqref{S2eq20op} together with the fact that
\beq \label{S2eq20yp}
\sup_{y\in\R^+}\bigl(e^{-\f{3-\ga}2\Psi}\int_0^ye^{\f{3-\ga}2\Psi}\,dy'\bigr)\leq C\w{t}^{\f12},
\eeq
implies \eqref{S2eq20}.

Finally, let us turn to the proof of \eqref{S7eq21}. We first observe from \eqref{S2eq17} that
\beq\label{S7eq20}
\p_y(y\vf)=\Bigl(1-\f{y^2}{2\w{t}}\Bigr)\vf+yG.
\eeq
Then along the same line to the proof of \eqref{S2eq20op}, we deduce that
\beno
\begin{split}
\bigl\|&e^{\ga\Psi}\bigl(1-\f{1}{2}\w{t}^{-1}y^2\bigr)\D_k^\h\vf_\Phi(t)\bigr\|_{L^2_+}\\
=&\bigl\|e^{(\ga-2)\Psi}\bigl(1-\f{1}{2}\w{t}^{-1}y^2\bigr)\int_0^ye^{2\Psi}\int_{y'}^\infty\D_k^\h\p_yG_\Phi\,dz\,dy'\bigr\|_{L^2_+}\\
\leq &\bigl\|e^{(\ga-2)\Psi}\bigl(1-\f{1}{2}\w{t}^{-1}y^2\bigr)\int_0^ye^{2\Psi}\Bigl(\int_{y'}^\infty e^{-2\Psi}\,dz\Bigr)^{\f12}\Bigl(\int_0^\infty
|e^\Psi\D_k^\h\p_yG_\Phi|^2\,dy\Bigr)^{\f12}\,dy'\bigr\|_{L^2_+}\\
\lesssim &\w{t}^{\f14}\bigl\|e^{(\ga-2)\Psi}\bigl(1-\f{1}{2}\w{t}^{-1}y^2\bigr)\int_0^ye^{\f{3-\ga}2\Psi}\Bigl(\int_0^\infty
|e^\Psi\D_k^\h\p_yG_\Phi|^2\,dy\Bigr)^{\f12}\,dy'\bigr\|_{L^2_+}\\
\lesssim &\w{t}\|e^\Psi\D_k^\h \p_yG_\Phi\|_{L^2_+}.
\end{split}
\eeno
While a direct computation ensures that
\beno
\begin{split}
\bigl\|e^{\ga\Psi}y\D_k^\h G_\Phi(t)\bigr\|_{L^2_+}\lesssim &  \bigl\|e^{\ga\Psi} y\Bigl(\int_0^ye^{-2\Psi}\,dy'\Bigr)^{\f12}\Bigl(\int_0^\infty |e^{\Psi} \D_k^\h\p_yG_\Phi|^2\,dy\Bigr)^{\f12}\bigr\|_{L^2_+}\\
\lesssim &\w{t}^{\f14}\bigl\|y e^{\f{\ga-1}2\Psi}\Bigl(\int_0^\infty |e^{\Psi} \D_k^\h\p_yG_\Phi|^2\,dy\Bigr)^{\f12}\bigr\|_{L^2_+}\\
\lesssim &\w{t}\|e^\Psi\D_k^\h \p_yG_\Phi\|_{L^2_+}.
\end{split}
\eeno
This along  with \eqref{S7eq20} ensures that
\beq \label{S7eq21yu}
\|e^{\ga\Psi}\D_k^\h\p_y(y\vf)_\Phi(t)\|_{L^2_+}
\lesssim \w{t}\|e^\Psi\D_k^\h \p_yG_\Phi\|_{L^2_+}.
\eeq

By exactly the same procedure as  that in the proof of \eqref{S7eq21yu}, we find
\beno
\|e^{\ga\Psi}\D_k^\h\p_y(y\vf)_\Phi(t)\|_{L^\infty_{\rm v}(L^2_\h)}
\lesssim \w{t}^{\f34}\|e^\Psi\D_k^\h \p_yG_\Phi\|_{L^2_+}.
\eeno
This together with \eqref{S7eq21yu}  ensures \eqref{S7eq21}. We thus conclude  the proof of  Lemma \ref{S0lem1}.
\end{proof}

\smallskip

\renewcommand{\theequation}{\thesection.\arabic{equation}}
\setcounter{equation}{0}
\section{The decay-in-time energy  estimate of $G^s$}\label{sect1a}

The goal of this section is to present the proof of Proposition \ref{prop2.1}.  Especially, we are going
to prove in the classical weighted energy space that $G^s$ determined by \eqref{S2eq7} decays faster than $u^s,$ with
the decay rate being given by \eqref{Soeq1}.
We start the proof of Proposition \ref{prop2.1} by the following lemma:

\begin{lem}\label{lem2.2}
{\sl Let  $G^s(t,y)$ and  $\Psi(t,y)$ be defined  respectively by  \eqref{S2eq4}  and \eqref{eq2.7}.
  Then  for any $t>0,$ one has
\beq \label{S2eq5}
\bigl\|\w{t'}^{\f54}e^\Psi G^s\bigr\|_{L^\infty(\R^+; L^2_{\rm v})}\leq C\e\|\w{t'}^{\f54}H\|_{L^1(\R_+;L^2_{\rm v})},
\eeq
and
\beq \label{S2eq6}
\begin{split}
\int_{\f{t}2}^t\bigl\|\w{t'}^{\f54}e^\Psi\p_y G^s(t')\bigr\|_{L^2_{\rm v}}^2\,dt'\lesssim & \e^2\bigl(\|\w{t'}^{\f54}H\|_{L^1(\R_+;L^2_{\rm v})}^2+\|\w{t'}^{\f74} H\|_{L^2(\R_+;L^2_{\rm v})}^2\bigr),
\end{split}
\eeq for $H$ given by \eqref{S2eq7}.
}
\end{lem}

\begin{proof}
By taking $L^2_{\rm v}$ inner product  of  the $G^s$ equation  of \eqref{S2eq7} with $e^{2\Psi}G^s,$ we obtain
\beno
\left(\p_tG^s | e^{2\Psi}G^s\right)_{L^2_{\rm v}}-\left(\p_y^2G^s | e^{2\Psi}G^s\right)_{L^2_{\rm v}}+\w{t}^{-1}\bigl\|e^\Psi G^s(t)\bigr\|_{L^2_{\rm v}}^2=\e \left( H | e^{2\Psi} G^s\right)_{L^2_{\rm v}}.
\eeno
It is easy to observe that
\beno
\left(\p_tG^s | e^{2\Psi}G^s\right)_{L^2_{\rm v}}=\f12\f{d}{dt}\bigl\|e^\Psi G^s(t)\bigr\|_{L^2_{\rm v}}^2
-\int_{\R^+}e^{2\Psi}\p_t\Psi|G^s|^2\,dy.
\eeno
Due to $G^s|_{y=0}=0,$ we get, by using integration by parts and Young's inequality, that
\beno
\begin{split}
-\left(\p_y^2G^s | e^{2\Psi}G^s\right)_{L^2_{\rm v}}=&\bigl\|e^\Psi\p_y G^s\bigr\|_{L^2_{\rm v}}^2+2\int_{\R^+}e^{2\Psi}\p_y\Psi\p_yG^s G^s\,dy\\
\geq &\f12\bigl\|e^\Psi\p_yG^s\bigr\|_{L^2_{\rm v}}^2-2\int_{\R^+}e^{2\Psi}(\p_y\Psi)^2|G^s|^2\,dy.
\end{split}
\eeno
As a result, thanks to \eqref{eq2.8}, we obtain
\beq \label{S2eq8}
\f12\f{d}{dt}\bigl\|e^\Psi G^s(t)\bigr\|_{L^2_{\rm v}}^2+\f12\bigl\|e^\Psi\p_y G^s(t)\bigr\|_{L^2_{\rm v}}^2 +\w{t}^{-1}\bigl\|e^\Psi G^s(t)\bigr\|_{L^2_{\rm v}}^2
\leq
\e \bigl\|e^\Psi G^s(t)\|_{L^2_{\rm v}}\bigl\|e^\Psi H(t)\bigr\|_{L^2_{\rm v}}.
\eeq
Applying Lemma \ref{lem2.1} for $d=0$ yields
\beno
\bigl\|e^\Psi\p_y G^s(t)\bigr\|_{L^2_{\rm v}}^2\geq \f1{2\w{t}}\bigl\|e^\Psi G^s(t)\bigr\|_{L^2_{\rm v}}^2,
\eeno
so that we deduce from \eqref{S2eq8} that
\beno
\f12\f{d}{dt}\bigl\|e^\Psi G^s(t)\bigr\|_{L^2_{\rm v}}^2+\f5{4\w{t}}\bigl\|e^\Psi G^s(t)\bigr\|_{L^2_{\rm v}}^2\leq
\e \bigl\|e^\Psi G^s(t)\|_{L^2_{\rm v}}\bigl\|e^\Psi H(t)\bigr\|_{L^2_{\rm v}},
\eeno
which implies
\beno
\f{d}{dt}\bigl\|e^\Psi G^s(t)\bigr\|_{L^2_{\rm v}}+\f5{4\w{t}}\bigl\|e^\Psi G^s(t)\bigr\|_{L^2_{\rm v}}\leq
\e \bigl\|e^\Psi H(t)\|_{L^2_{\rm v}},
\eeno
and
\beno
\f{d}{dt}\left(\w{t}^{\f54}\bigl\|e^\Psi G^s(t)\bigr\|_{L^2_{\rm v}}\right)\leq
\e \w{t}^{\f54}\bigl\|e^\Psi H(t)\|_{L^2_{\rm v}}.
\eeno
Integrating the above inequality over $[0,t]$ gives rise to \eqref{S2eq5}.

On the other hand, we deduce from \eqref{S2eq8} and Young's inequality that
\beq \label{S2eq89}
\begin{split}
\f{d}{dt}\bigl\|e^\Psi G^s(t)\bigr\|_{L^2_{\rm v}}^2+\bigl\|e^\Psi\p_y G^s(t)\bigr\|_{L^2_{\rm v}}^2+&2\w{t}^{-1}\bigl\|e^\Psi G^s(t)\bigr\|_{L^2_{\rm v}}^2\\
\leq &
2\e\w{t}^{\f12}\bigl\|e^\Psi H(t)\|_{L^2_{\rm v}}\w{t}^{-\f12}\bigl\|e^\Psi G^s(t)\bigr\|_{L^2_{\rm v}}\\
\leq &
\e^2\w{t}\bigl\|e^\Psi H(t)\|_{L^2_{\rm v}}^2+\w{t}^{-1}\bigl\|e^\Psi G^s(t)\bigr\|_{L^2_{\rm v}}^2.
\end{split}
\eeq
Multiplying the above inequality by $\w{t}^{\f52}$ and then integrating the resulting inequality over $[t/2,t],$
we obtain
\beno
\begin{split}
\int_{\f{t}2}^t\bigl\|\w{t'}^{\f54}e^\Psi\p_y G^s(t')\bigr\|_{L^2_{\rm v}}^2\,dt'\leq &
\bigl\|\w{t/2}^{\f54}e^\Psi G^s(t/2)\bigr\|_{L^2_{\rm v}}^2\\
&+\f52\int_{\f{t}2}^t\w{t'}^{\f32}\bigl\|e^\Psi G^s(t')\bigr\|_{L^2_{\rm v}}^2\,dt'
+\e^2\int_{\f{t}2}^t\w{t'}^{\f72}\bigl\|e^\Psi H(t')\|_{L^2_{\rm v}}^2\,dt'\\
\leq &\max_{t'\in[0,t]}\bigl\|\w{t'}^{\f54}e^\Psi G^s(t')\bigr\|_{L^2_{\rm v}}^2\bigl(1+\f{5\ln2}2\bigr)
+\e^2\bigl\|\w{t'}^{\f74}e^\Psi H\bigr\|_{L^2_t(L^2_{\rm v})}^2.
\end{split}
\eeno
Inserting \eqref{S2eq5} into the above inequality leads to \eqref{S2eq6}.
This finishes the proof of Lemma \ref{lem2.2}. \end{proof}

\begin{rmk} By integrating \eqref{S2eq89} over $[0,t],$ we obtain
\beq\label{Gse}
\bigl\|e^\Psi\p_y G^s\bigr\|_{L^2_t(L^2_{\rm v})}^2
\leq
\e^2\int_0^\infty\w{t}\bigl\|e^\Psi H(t)\|_{L^2_{\rm v}}^2\,dt.
\eeq
\end{rmk}

Let us now present the proof of Proposition \ref{prop2.1}.

\begin{proof}[Proof of Proposition \ref{prop2.1}]  In view of \eqref{S1eq3} and \eqref{S2eq3}, both $m$ and $M$
are supported in $[0,2]$ for any $t\geq 0,$ so that  we  observe from \eqref{S2eq7} that
\beno
\begin{split} \|\w{t}^{\f54}H\|_{L^1(\R_+;L^2_{\rm v})}\lesssim
&\|\w{t}^{\f14} yM\|_{L^1(\R_+;L^2_{\rm v})}+\|\w{t}^{\f54}m\|_{L^1(\R_+;L^2_{\rm v})}\\
\leq & C\int_0^\infty\w{t}^{\f54}\left(|f(t)|+|f'(t)|\right)\,dt\leq CC_f,
\end{split}
\eeno
and
\beno
\begin{split}
\|\w{t}^{\f74} H\|_{L^2(\R_+;L^2_{\rm v})}\lesssim
&\|\w{t}^{\f34} yM\|_{L^1(\R_+;L^2_{\rm v})}+\|\w{t}^{\f74}m\|_{L^1(\R_+;L^2_{\rm v})}\\
\leq & C\Bigl(\int_0^\infty\w{t}^{\f72}\left(f^2(t)+(f'(t))^2\right)\,dt\Bigr)^{\f12}\leq CC_f,
\end{split}
\eeno
for $\cC_{f}$ given by \eqref{S2eq12}. Hence, for any $t>0,$ we deduce from \eqref{S2eq6}  and \eqref{Gse} that
\beq\label{S2eq10} \bigl\|e^\Psi\p_y G^s\bigr\|_{L^2_t(L^2_{\rm v})}^2+
\int_{\f{t}2}^t\bigl\|\w{t'}^{\f54}e^\Psi\p_y G^s(t')\bigr\|_{L^2_{\rm v}}^2\,dt'\leq CC_f^2\e^2.
\eeq

While for any $t>1,$ we fix an integer $N_t$ so that $2^{N_t-1}\leq t<2^{N_t},$ which implies $t/2<2^{N_t-1}.$ Then we deduce from \eqref{S2eq10} that
\beno
\begin{split}
\int_{2^{N_t-1}}^t\w{t'}^{\f14}\|e^\Psi\p_y G^s(t')\|_{L^2_{\rm v}}\,dt'\leq &\Bigl(\int_{2^{N_t-1}}^t \w{t'}^{-2}\,dt'\Bigr)^{\f12} \Bigl(\int_{t/2}^t \w{t'}^{\f52}\bigl\|e^\Psi\p_y G^s(t')\bigr\|_{L^2_{\rm v}}^2\,dt'\Bigr)^{\f12}\\
\leq &C 2^{-\f{N_t}2}\cC_{f}\e.
\end{split}
\eeno
Along the same line, for any $j\in [0, N_t-2],$ we have
\beno
\int_{2^{j}}^{2^{j+1}}\w{t'}^{\f14}\|e^\Psi\p_yG^s(t')\|_{L^2_{\rm v}}\,dt'\leq C 2^{-\f{j}2}\cC_{f}\e.
\eeno
As a a result,   it comes out
\beno
\begin{split}
\int_{0}^t&\w{t'}^{\f14}\|e^\Psi\p_yG^s(t')\|_{L^2_{\rm v}}\,dt'\leq 2^{\f14}\int_{0}^1\|e^\Psi\p_yG^s(t')\|_{L^2_{\rm v}}\,dt'\\
&+\int_{2^{N_t-1}}^t\w{t'}^{\f14}\|e^\Psi\p_yG^s(t')\|_{L^2_{\rm v}}\,dt'+\sum_{j=0}^{N_t-2}\int_{2^{j}}^{2^{j+1}}\w{t'}^{\f14}\|e^\Psi\p_yG^s(t')\|_{L^2_{\rm v}}\,dt'\\
\leq &C\cC_{f}\e\Bigl(1+\sum_{j=0}^\infty 2^{-\f{j}2}\Bigr)\leq C\cC_{f}\e.
\end{split}
\eeno
 This completes the proof of Proposition \ref{prop2.1}.
\end{proof}

Motivated by the proof of Lemma \ref{S0lem1}, we have the following corollary of Proposition \ref{prop2.1}:

\begin{col}\label{S2col1}
{\sl Let $u^s$  be determined   by \eqref{S2eq7up}.
Then for any $\ga\in (0,1),$  we have
\beq \label{S2eq11}
\int_0^\infty\w{t}^{\f14}\bigl\|e^{\ga\Psi}\p_yu^s(t')\|_{L^2_{\rm v}}\,dt'\leq C\cC_f\e.
\eeq
}
\end{col}

\begin{proof} In view of \eqref{S2eq7up}, we have
\beno
u^s(t,y)=\p_y\psi^s(t,y)=-\f{y}{2\w{t}}e^{-\f{y^2}{4\w{t}}}\int_0^ye^{\f{(y')^2}{4\w{t}}}G^s(t,y')\,dy'+G^s(t,y)
\eeno
and
\beq\label{S2eq13}
\begin{split}
\p_yu^s(t,y)=\p_y^2\psi^s(t,y)=&\Bigl(-\f{1}{2\w{t}}+\f{y^2}{4\w{t}^2}\Bigr)e^{-\f{y^2}{4\w{t}}}\int_0^ye^{\f{(y')^2}{4\w{t}}}G^s(t,y')\,dy'\\
&-\f{y}{2\w{t}}G^s(t,y)+\p_yG^s(t,y).
\end{split}
\eeq
Since $\lim_{y\to\infty} G^s(t,y)=0,$ we write $G^s(t,y)=-\int_y^\infty \p_yG^s(t,y')\,dy'.$ Then due to $\ga\in (0,1),$
we find
\beno
\begin{split}
|G^s(t,y)|\leq &\Bigl(\int_y^\infty e^{-2\Psi}\,dy'\Bigr)^{\f12}\|e^\Psi \p_yG^s(t)\|_{L^2_{\rm v}}\\
\lesssim & \w{t}^{\f14}e^{-\f{1+\ga}2\Psi}\|e^\Psi \p_yG^s(t)\|_{L^2_{\rm v}},
\end{split}
\eeno
and
\beno
\begin{split}
\bigl|\int_0^ye^{\f{(y')^2}{4\w{t}}}G^s(t,y')\,dy'\bigr|
\lesssim &\w{t}^{\f14}\int_0^ye^{\f{3-\ga}2\Psi}\,dy'\|e^\Psi \p_yG^s(t)\|_{L^2_{\rm v}}.
\end{split}
\eeno
Hence by virtue of \eqref{S2eq13}, we infer
\beno
\begin{split}
\|e^{\ga\Psi}\p_yu^s(t)\|_{L^2_{\rm v}}\lesssim &\Bigl(\w{t}^{-\f34}\bigl\|e^{(\ga-2)\Psi} \int_0^ye^{\f{3-\ga}2\Psi}\,dy'\bigr\|_{L^2_{\rm v}}
+\w{t}^{-\f74}\bigl\|y^2e^{(\ga-2)\Psi} \int_0^ye^{\f{3-\ga}2\Psi}\,dy'\bigr\|_{L^2_{\rm v}}\\
&+\w{t}^{-\f34}\bigl\|ye^{-\f{1-\ga}2\Psi}\bigr\|_{L^2_{\rm v}}+1\Bigr)
\bigl\|e^{\Psi}\p_y G^s(t)\|_{L^2_{\rm v}},
\end{split}
\eeno
which together with \eqref{S2eq20yp} ensures that
\beq \label{S2eq13a}
\begin{split}
\|e^{\ga\Psi}\p_yu^s(t)\|_{L^2_{\rm v}}
\lesssim &\Bigl(\w{t}^{-\f14}\bigl\|e^{-\f{1-\ga}2\Psi} \bigr\|_{L^2_{\rm v}}
+\w{t}^{-\f54}\bigl\|y^2e^{-\f{1-\ga}2\Psi}\bigr\|_{L^2_{\rm v}}\\
&+\w{t}^{-\f34}\bigl\|ye^{-\f{1-\ga}2\Psi}\bigr\|_{L^2_{\rm v}}+1\Bigr)
\bigl\|e^{\Psi}\p_y G^s(t)\|_{L^2_{\rm v}}\\
\lesssim &\bigl\|e^{\Psi}\p_y G^s(t)\|_{L^2_{\rm v}},
\end{split}
\eeq
from which and \eqref{S2eq9}, we conclude the proof of  \eqref{S2eq11}.
\end{proof}
\smallskip

\renewcommand{\theequation}{\thesection.\arabic{equation}}
\setcounter{equation}{0}

\section{Analytic energy estimate to the primitive function of $u$
}\label{Sect4}

The goal of this section is to present the {\it a priori} weighted analytic  energy estimate to the primitive
function  $\vf$ to the solution of \eqref{S1eq5}, namely,
the proof of Proposition \ref{S4prop2}. The key ingredient lies in the following proposition:

\begin{prop}\label{S4prop1}
{\sl Let $\vf$ be a smooth enough solution of \eqref{S1eq2}. Let $\Phi(t,\xi)$ and $\Psi(t,y)$ be given by \eqref{eq2.6} and
\eqref{eq2.7} respectively.   Then for any nonnegative and non-decreasing function $h\in C^1(\R_+),$ there exits a large enough constant $\lambda$ so that
 \beq \label{S4eq13}
\begin{split} &\|\hbar^{\f12}e^\Psi\D_k^{\rm
h}\vf_\Phi\|_{L^\infty(t_0,t;L^2_+)}^2+2c\lam
2^k\int_{t_0}^t\dot{\tht}(t')\|\hbar^{\f12}\D_k^{\rm
h}\vf_\Phi(t')\|_{L^2_+}^2\,dt'\\
&\qquad+\|\hbar^{\f12}e^\Psi\D_k^{\rm
h}\p_y\vf_\Phi\|_{L^2(t_0,t;L^2_+)}^2\leq \|\hbar^{\f12}e^{\Psi}\D_k^{\rm
h}\vf_\Phi(t_0)\|_{L^2_+}^2\\
&+\int_{t_0}^t\hbar'(t')\|e^\Psi\D_k^{\rm
h}\vf_\Phi(t')\|_{L^2_+}^2\,dt'+Cd_k^22^{-k} \|\hbar^{\f12} e^\Psi
\vf_\Phi\|_{\wt{L}^2_{t_0,t;\dot{\tht}(t)}(\cB^{1,0})}^2, \end{split}
\eeq for any $t_0\in  [0,t]$
with $t<T^\ast,$ which is defined by \eqref{1.8a}.  }
\end{prop}

\begin{proof} In view of \eqref{S1eq2} and \eqref{eq2.4}, we write
 \beq \label{S4eq13av}
\begin{split}
\p_t\vf_\Phi-\p_{yy}\vf_\Phi&+\lam\dot{\tht}(t)|D_h|\vf_\Phi+\left[(u+u^s+\e f(t)\chi(y))\p_x \vf
\right]_\Phi\\
&\qquad\qquad+2\int_y^\infty \bigl[\p_y(u+u^s+\e f(t)\chi(y'))\p_x\vf\bigr]_\Phi
\,dy' =0.
\end{split}
\eeq
Here and in all that follows, we shall always denote $|D_\h|$ to be the Fourier multiplier
in the $x$ variable with symbol $|\xi|.$

By applying the dyadic operator
$\D_k^{\rm h}$ to \eqref{S4eq13av} and then taking the $L^2_+$ inner
product of the resulting equation with $\hbar(t)e^{2\Psi}\D_k^{\rm
h}\vf_\Phi,$  we find \beq \label{S4eq2}
\begin{split}
\hbar(t)&\bigl(e^\Psi\D_k^{\rm h}\left(\p_t\vf_\Phi-\p_{yy}\vf_\Phi\right)\ |\ e^\Psi\D_k^{\rm
h}\vf_\Phi\bigr)_{L^2_+}+\lam\dot{\tht}(t)\hbar(t)\bigl(e^\Psi|D_h|\D_k^{\rm
h}\vf_\Phi\ |\ e^\Psi\D_k^{\rm
h}\vf_\Phi\bigr)_{L^2_+}\\
&+\hbar(t)\bigl(e^\Psi\D_k^{\rm h}[(u+u^s+\e f(t)\chi(y))\p_x \vf
]_\Phi\ |\ e^\Psi\D_k^{\rm
h}\vf_\Phi\bigr)_{L^2_+}\\
&+2\hbar(t)\bigl(e^\Psi\int_y^\infty \D_k^{\rm h}\bigl[\p_y(u+u^s+\e f(t)\chi(y'))\p_x\vf\bigr]_\Phi
\,dy' \ |\ e^\Psi\D_k^{\rm
h}\vf_\Phi\bigr)_{L^2_+}=0.
\end{split}
\eeq
 In the rest of this section, we shall always assume that $t<T^\ast$ with
$T^\ast$ being determined by \eqref{1.8a} so that by
virtue of \eqref{eq2.6}, for any $t<T^\ast,$ there holds the
 convex inequality \eqref{1.8bb}.

Then the proof of Proposition \ref{S4prop1} relies on the following lemmas:

\begin{lem}\label{S4lem1}
{\sl Under the assumptions of Proposition \ref{S4prop1}, for any $t_0\in [0,t]$ with $t<T^\ast,$ we have
\beq
\label{S4eq12ad}
\begin{split}
\int_{t_0}^t&\hbar(t')\bigl(e^\Psi\D_k^{\rm h}\left(\p_t\vf_\Phi-\p_{yy}\vf_\Phi\right)\ |\ e^\Psi\D_k^{\rm
h}\vf_\Phi\bigr)_{L^2_+}\,dt'\\
\geq & \f12\Bigl(\|\hbar^{\f12}e^\Psi\D_k^{\rm
h}\vf_\Phi(t)\|_{L^2_+}^2-\|\hbar^{\f12}e^{\Psi}\D_k^{\rm
h}\vf_\Phi(t_0)\|_{L^2_+}^2\\
&\quad-\int_{t_0}^t\hbar'(t')\|e^\Psi\D_k^{\rm
h}\vf_\Phi(t')\|_{L^2_+}^2\,dt'+\|e^\Psi\D_k^{\rm
h}\p_y\vf_\Phi\|_{L^2(t_0,t;L^2_+)}^2\Bigr).
  \end{split} \eeq}
\end{lem}

\begin{lem}\label{S4lem2}
{\sl Under the assumptions of Proposition \ref{S4prop1}, for any $t_0\in [0,t]$ with $t<T^\ast,$ we have
\beq
\label{S4eq7a}
\begin{split}
\int_{t_0}^t\hbar(t')\bigl|\bigl(e^\Psi\D_k^{\rm
h}[(u+u^s+\e f(t')\chi(y))\p_x\vf]_\Phi\ |&\ e^\Psi\D_k^{\rm h}\vf_\Phi\bigr)_{L^2_+}\bigr|\,dt'\\
\lesssim & d_k^22^{-k}\|\hbar^{\f12}e^\Psi
\vf_\Phi\|_{\wt{L}^2_{t_0,t;\dot{\tht}(t)}(\cB^{1,0})}^2.  \end{split} \eeq}
\end{lem}

\begin{lem}\label{S4lem3}
{\sl Under the assumptions of Proposition \ref{S4prop1}, for any $t_0\in [0,t]$ with $t<T^\ast,$ we have
\beq
\label{S4eq9}
\int_{t_0}^t\hbar(t')\bigl|\bigl(e^\Psi\int_y^\infty \D_k^{\rm
h}[\p_yu\p_x\vf]_\Phi
\,dy' \ |\ e^\Psi\D_k^{\rm h}\vf_\Phi\bigr)_{L^2_+}\bigr|\,dt'\lesssim  d_k^22^{-k}\|\hbar^{\f12}e^\Psi
\vf_\Phi\|_{\wt{L}^2_{t_0,t;\dot{\tht}(t)}(\cB^{1,0})}^2. \eeq}
\end{lem}

Let us admit the above lemmas for the time being and continue our proof of Proposition \ref{S4prop1}.

Indeed it follows from Lemma \ref{lem:Bern} that
\beq \label{S4eq40} \lam\dot{\tht}(t)\bigl(e^\Psi|D_h|\D_k^{\rm h}\vf_\Phi\ |\
e^\Psi\D_k^{\rm h}\vf_\Phi\bigr)_{L^2_+}\geq
c\lam\dot{\tht}(t)2^k\|e^\Psi\D_{k}^{\rm h}\vf_\Phi(t)\|_{L^2_+}^2.
\eeq
While it is easy to observe that
\beno
\begin{split}
\int_{t_0}^t&\hbar(t')\bigl|\bigl(e^\Psi\int_y^\infty \D_k^{\rm h}\bigl[\p_y(u^s+\e f(t')\chi(y'))\p_x\vf\bigr]_\Phi
\,dy' \ |\ e^\Psi\D_k^{\rm
h}\vf_\Phi\bigr)_{L^2_+}\bigr|\,dt'\\
&\lesssim \int_{t_0}^t\hbar(t')\int_{\R^2_+}e^{-\f34\Psi}\int_y^\infty e^{\f74\Psi}
|\p_y(u^s+\e f(t')\chi(y'))||\p_x\D_k^{\rm h}\vf_\Phi|
\,dy' |e^\Psi\D_k^{\rm
h}\vf_\Phi|\,dx\,dy\,dt'\\
&\lesssim 2^k\int_{t_0}^t\hbar(t')\|e^{-\f34\Psi}\|_{L^2_{\rm v}}\bigl(\|e^{\f34\Psi}
\p_yu^s\|_{L^2}+\e f(t')\|e^{\Psi}\p_y\chi\|_{L^2_{\rm v}}\bigr)\|e^\Psi\D_k^{\rm h}\vf_\Phi\|_{L^2_+}^2
\,dt',
\end{split}
\eeno
from which,  \eqref{1.9}, \eqref{S2eq13a}  and  Definition \ref{def1.1}, we infer
\beq \label{S4eq9a}
\begin{split}
\int_{t_0}^t\hbar(t')\bigl|\bigl(e^\Psi\int_y^\infty &\D_k^{\rm h}\bigl[\p_y(u^s+\e f(t')\chi(y'))\p_x\vf\bigr]_\Phi
\,dy' \ |\ e^\Psi\D_k^{\rm
h}\vf_\Phi\bigr)_{L^2_+}\bigr|\,dt'\\
&\lesssim 2^k\int_0^t\dot{\theta}(t')\|\hbar^{\frac12}e^\Psi\D_k^{\rm h}\vf_\Phi(t')\|_{L^2_+}^2
\,dt'\lesssim  d_k^22^{-k}\|\hbar^{\f12}e^\Psi
\vf_\Phi\|_{\wt{L}^2_{t_0,t;\dot{\tht}(t)}(\cB^{1,0})}^2.
\end{split}
\eeq

By integrating \eqref{S4eq2} over
$[t_0,t]$ and then
inserting the estimates,  (\ref{S4eq12ad}-\ref{S4eq9a}) into the resulting inequality, we
obtain \eqref{S4eq13}. This completes the proof of Proposition \ref{S4prop1}. \end{proof}

With Proposition \ref{S4prop1}, we now present the proof of Proposition \ref{S4prop2}

\begin{proof}[Proof of Proposition \ref{S4prop2}] we first observe from  \eqref{S2eq1} that
\beno
\begin{split}
\f12&\int_0^t\w{t'}^{-\f12}\|e^\Psi\D_k^{\rm
h}\vf_\Phi\|_{L^2_+}^2\,dt'\leq \int_0^t\|\w{t'}^{\f14}e^\Psi\D_k^{\rm
h}\p_y\vf_\Phi\|_{L^2_+}^2\,dt'.
\end{split}
\eeno
So that by taking $t_0=0$ and $\hbar(t)=\w{t}^{\frac12}$ in \eqref{S4eq13},
we  obtain
\beno
\begin{split} \|\w{t'}^{\f14}e^\Psi\D_k^{\rm
h}\vf_\Phi\|_{L^\infty_t(L^2_+)}^2&+2c\lam
2^k\int_{0}^t\dot{\tht}(t')\|\w{t'}^{\f14}\D_k^{\rm
h}\vf_\Phi(t')\|_{L^2_+}^2\,dt'\\
&\leq \bigl\|e^{\f{y^2}8}e^{\delta|D_x|}\D_k^{\rm
h}\vf_0\bigr\|_{L^2_+}^2+Cd_k^22^{-k} \|\w{t}^{\f14} e^\Psi
\vf_\Phi\|_{\wt{L}^2_{t,\dot{\tht}(t)}(\cB^{1,0})}^2. \end{split}
\eeno
By taking square root of the above inequality and then multiplying
the resulting one by $2^{\f{k}2}$ and finally summing
over $k\in\Z,$ we find for any $t< T^\ast$ \beq
\label{S4eq30}\begin{split} &\|\w{t'}^{\f14}e^\Psi
\vf_\Phi\|_{\wt{L}^\infty_t(\cB^{\f{1}2,0})}+\sqrt{2c\lam}\|\w{t'}^{\f14} e^\Psi
\vf_\Phi\|_{\wt{L}^2_{t,\dot{\tht}(t)}(\cB^{1,0})}\\
&\qquad\qquad\qquad\qquad\qquad\leq \|e^{\f{y^2}8}
e^{\delta|D_x|}\vf_0\|_{\cB^{\f{1}2,0}}+\sqrt{C}\|\w{t'}^{\f14} e^\Psi
\vf_\Phi\|_{\wt{L}^2_{t,\dot{\tht}(t)}(\cB^{1,0})}. \end{split}
\eeq By  taking $\lam$ in \eqref{S4eq30} to be so large  that
$c\lam\geq C,$ we achieve \eqref{S4eq11ag}.

On the other hand, in view of \eqref{S4eq13}, we get, by using a similar derivation of \eqref{S4eq30}, that
\beno
\begin{split}
&\|\hbar^{\f12}e^{\Psi}\vf_\Phi\|_{\wt{L}^\infty(t_0,t;\cB^{\f12,0})}+\sqrt{2c\lam}
\|\hbar^{\f12}e^{\Psi}\vf_\Phi\|_{\wt{L}^2_{t_0,t;\dot{\theta}(t)}(\cB^{1,0})}+
\|\hbar^{\f12}e^{\Psi}\p_y\vf_\Phi\|_{\wt{L}^2(t_0,t;\cB^{\f12,0})}\\
&\qquad\qquad\leq \|\hbar^{\f12}e^{\Psi}\vf_\Phi(t_0)\|_{\cB^{\f12,0}}+\bigl\|\sqrt{\hbar'}e^{\Psi}\vf_\Phi\|_{\wt{L}^2(t_0,t;\cB^{\f12,0})}
+\sqrt{C}\|\hbar^{\f12}e^{\Psi}\vf_\Phi\|_{\wt{L}^2_{t_0,t;\dot{\theta}(t)}(\cB^{1,0})}.
\end{split}
\eeno
Taking $c\lam\geq C$ in the above inequality gives rise to \eqref{S4eq16}.
 This concludes the proof of Proposition \ref{S4prop2}. \end{proof}

Let us end this section by  the proofs of Lemmas \ref{S4lem1}-\ref{S4lem3}.

\begin{proof}[Proof of Lemma \ref{S4lem1}]
We first get, by using integration by parts, that
\beno
\begin{split}
\bigl(e^\Psi\p_t\D_k^{\rm h}\vf_\Phi\ |\ e^\Psi\D_k^{\rm
h}\vf_\Phi\bigr)_{L^2_+}=\bigl(\p_t(e^\Psi\D_k^{\rm h}\vf_\Phi)\ |\
e^\Psi\D_k^{\rm h}\vf_\Phi\bigr)_{L^2_+}-\bigl(\pa_t\Psi e^\Psi\D_k^{\rm
h}\vf_\Phi\ |\ e^\Psi\D_k^{\rm h}\vf_\Phi\bigr)_{L^2_+}.
\end{split}
\eeno
By multiplying the above equality by $\hbar(t)$ and then integrating the resulting one over $[t_0,t],$ we find
\beq\label{S4eq22}
\begin{split}
\int_{t_0}^t&\hbar(t')\bigl(e^\Psi\p_t\D_k^{\rm h}\vf_\Phi\ |\ e^\Psi\D_k^{\rm
h}\vf_\Phi\bigr)_{L^2_+}\,dt'\\
=&\f12\|\hbar^{\f12}e^\Psi\D_k^{\rm
h}\vf_\Phi(t)\|_{L^2_+}^2-\f12\|\hbar^{\f12}e^{\Psi}\D_k^{\rm
h}\vf_\Phi(t_0)\|_{L^2_+}^2\\
&-\f12\int_{t_0}^t\hbar'(t')\|e^\Psi\D_k^{\rm
h}\vf_\Phi(t')\|_{L^2_+}^2\,dt'-\int_{t_0}^t\int_{\R^2_+}\hbar\pa_t\Psi|e^\Psi\D_k^{\rm
h}\vf_\Phi|^2\,dx\,dy\,dt'.
\end{split}
\eeq

Whereas due to $\p_y\vf|_{y=0}=0,$  by using integration by parts and Young's inequality, we achieve
 \beno
\begin{split}
&-\int_{t_0}^t\bigl(e^\Psi\p_{yy}\D_k^{\rm h}\vf_\Phi\ |\ e^\Psi\D_k^{\rm
h}\vf_\Phi\bigr)_{L^2_+}\,dt'\\
&\quad=\|e^\Psi\D_k^{\rm
h}\p_y\vf_\Phi\|_{L^2(t_0,t;L^2_+)}^2+2\int_{t_0}^t\int_{\R^2_+}\pa_y\Psi e^{2\Psi}\D_k^{\rm
h}\vf_\Phi\D_k^{\rm h}\p_y\vf_\Phi\,dx\,dy\,dt'\\
&\quad\geq \f12\|e^\Psi\D_k^{\rm
h}\p_y\vf_\Phi\|_{L^2(t_0,t;L^2_+)}^2-2\int_{t_0}^t\int_{\R^2_+}(\pa_y\Psi)^2|e^\Psi\D_k^{\rm
h} \vf_\Phi|^2\,dx\,dy\,dt',
\end{split}
\eeno
which together with \eqref{eq2.8} and \eqref{S4eq22} ensures \eqref{S4eq12ad}. This finishes the proof of Lemma  \ref{S4lem1}. \end{proof}

\begin{proof}[Proof of Lemma \ref{S4lem2}] By applying Bony's decomposition \eqref{Bony} in the horizontal variable to $u\p_x \vf,$ we write
\beno u\p_x\vf=T^{\rm h}_u\p_x \vf+T^{\rm h}_{\p_x\vf }u+R^{\rm h}(u,\p_x \vf). \eeno
Considering
\eqref{1.8bb} and the support properties to the Fourier transform of
the terms in $T^{\rm h}_u\p_x \vf,$ we write
$$\longformule{ \int_{t_0}^t\hbar(t')\bigl|\bigl(e^\Psi\D_k^{\rm h}\bigl[T^{\rm
h}_u \p_x \vf\bigr]_\Phi\ |\ e^\Psi\D_k^{\rm
h}\vf_\Phi\bigr)_{L^2_+}\bigr|\,dt'}{{}\lesssim \sum_{|k'-k|\leq
4}\int_{t_0}^t\|S_{k'-1}^{\rm
h}u_\Phi(t')\|_{L^\infty_+}\|\hbar^{\f12}e^{\Psi}\D_{k'}^{\rm h}\p_x
\vf_\Phi(t')\|_{L^2_+}\|\hbar^{\f12} e^{\Psi}\D_{k}^{\rm
h}\vf_\Phi(t')\|_{L^2_+}\,dt'.} $$
While it follows from \eqref{S2eq20} and \eqref{S2eq19p} that
\beq \label{S4eq3}
\begin{split}
\|e^{\f34\Psi}\D_k^{\rm h} u_\Phi(t')\|_{L^\infty_{\rm v}(L^2_{\rm h})}\lesssim &\bigl\|e^{\f34\Psi}\int_y^\infty\D_k^\h \p_yu_\Phi(t')\,dy'\bigr\|_{L^\infty_{\rm v}(L^2_{\rm h})}\\
\lesssim & \bigl\|e^{\f34\Psi}\Bigl(\int_y^\infty e^{-\f32\Psi}\,dy\Bigr)^{\f12}\bigr\|_{L^\infty_{\rm
v}}\|e^{\f34\Psi}\D_k^{\rm h}\p_yu_\Phi(t')\|_{L^2_+}\\
\lesssim &d_k(t')2^{-\frac{k}2}\w{t'}^{\f14}\|e^\Psi\p_yG_\Phi(t')\|_{\cB^{\f12,0}},
\end{split}
\eeq where $\bigl\{ d_k(t') \bigr\}_{k\in\Z}$ designates a non-negative generic element in the unit sphere of $\ell^1(\Z)$ for any $t'>0.$
Then we get, by
applying Lemma \ref{lem:Bern}, that \beno
\begin{split}
\|S_{k-1}^{\rm
h}u_\Phi(t')\|_{L^\infty_+}\lesssim &
\sum_{k'\leq k-2} 2^{\f{k}2}\|\D_k^\h
u_\Phi(t')\|_{L^\infty_{\rm v}(L^2_\h)}\\
\lesssim &\w{t'}^{\f14}\|e^\Psi\p_yG_\Phi(t')\|_{\cB^{\f12,0}},
\end{split}
\eeno which together with \eqref{1.9} ensures that \beno \|S_{k'-1}^{\rm h}u_\Phi(t')\|_{L^\infty_+}
\lesssim
 \dot{\tht}(t').  \eeno
Whence  in view of Definition \ref{def1.1}, by applying Lemma \ref{lem:Bern} and
H\"older's inequality, we obtain \beq\label{S4eq4}
\begin{split}
&\int_{t_0}^t\hbar(t')\bigl|\bigl(e^\Psi\D_k^{\rm h}\bigl[T^{\rm h}_u\p_x
\vf\bigr]_\Phi\ |\ e^\Psi\D_k^{\rm
h}\vf_\Phi\bigr)_{L^2_+}\bigr|\,dt'\\
&\lesssim \sum_{|k'-k|\leq
4}2^{k'}\Bigl(\int_{t_0}^t\dot{\tht}(t')\|\hbar^{\f12} e^\Psi\D_{k'}^{\rm
h}\vf_\Phi(t')\|_{L^2_+}^2\,dt'\Bigr)^{\f12}\Bigl(\int_{t_0}^t\dot{\tht}(t')\|\hbar^{\f12} e^\Psi\D_k^{\rm
h}\vf_\Phi(t')\|_{L^2_+}^2\,dt'\Bigr)^{\f12}\\
&\lesssim d_k^22^{-k}\|\hbar^{\f12}e^\Psi
\vf_\Phi\|_{\wt{L}^2_{t_0,t;\dot{\tht}(t)}(\cB^{1,0})}^2.
\end{split}
\eeq

Similarly, we have
$$\longformule{ \int_{t_0}^t\hbar(t')\bigl|\bigl(e^\Psi\D_k^{\rm h}\bigl[T^{\rm
h}_{\p_x \vf}u\bigr]_\Phi\ |\ e^\Psi\D_k^{\rm
h}\vf_\Phi\bigr)_{L^2_+}\bigr|\,dt'}{{}\lesssim \sum_{|k'-k|\leq
4}\int_{t_0}^t\|\hbar^{\f12}e^{\Psi} S_{k'-1}^{\rm
h}\p_x\vf_\Phi(t')\|_{L^2_{\rm v}(L^\infty_\h)}\|\D_{k'}^{\rm h}
u_\Phi(t')\|_{L^\infty_{\rm v}(L^2_\h)}\|\hbar^{\f12} e^{\Psi}\D_{k}^{\rm
h}\vf_\Phi(t')\|_{L^2_+}\,dt',} $$
from which and \eqref{S4eq3}, we infer
\beno
\begin{split}
&\int_{t_0}^t\hbar(t')\bigl|\bigl(e^\Psi\D_k^{\rm h}\bigl[T^{\rm
h}_{\p_x \vf}u\bigr]_\Phi\ |\ e^\Psi\D_k^{\rm
h}\vf_\Phi\bigr)_{L^2_+}\bigr|\,dt'\\
&\lesssim \sum_{|k'-k|\leq
4} 2^{-\frac{k'}2} \int_{t_0}^t\dot{\theta}(t')\|\hbar^{\f12} e^{\Psi} S_{k'-1}^{\rm
h}\p_x\vf_\Phi(t')\|_{L^2_{\rm v}(L^\infty_\h)}\|\hbar^{\f12} e^{\Psi}\D_{k}^{\rm
h}\vf_\Phi(t')\|_{L^2_+}\,dt'\\
&\lesssim \sum_{|k'-k|\leq
4}2^{-\frac{k'}2}\Bigl(\int_{t_0}^t\dot{\theta}(t')\|\hbar^{\f12}e^{\Psi} S_{k'-1}^{\rm
h}\p_x\vf_\Phi(t')\|_{L^2_{\rm v}(L^\infty_\h)}^2\,dt'\Bigr)^{\frac12}\\
&\qquad\qquad\qquad\qquad\qquad\qquad\times\Bigl(\int_{t_0}^t\dot{\theta}(t')
\|\hbar^{\f12} e^{\Psi}\D_{k}^{\rm
h}\vf_\Phi(t')\|_{L^2_+}^2\,dt'\Bigr)^{\frac12}.
\end{split}
\eeno
Yet it follows from Lemma \ref{lem:Bern} and Definition \ref{def1.1} that
\beq \label{S4eq8}
\begin{split}
\Bigl(\int_{t_0}^t&\dot{\theta}(t')\|\hbar^{\f12}e^{\Psi} S_{k'-1}^{\rm
h}\p_x\vf_\Phi(t')\|_{L^2_{\rm v}(L^\infty_\h)}^2\,dt'\Bigr)^{\frac12}\\
\lesssim &
\sum_{j\leq k'-2}2^{\frac{3j}2}\Bigl(\int_{t_0}^t\dot{\theta}(t')\|\hbar^{\f12}e^{\Psi}\D_{j}^{\rm
h}\vf_\Phi(t')\|_{L^2}^2\,dt'\Bigr)^{\frac12}\\
\lesssim &d_{k'}2^{\frac{k'}2}\|\hbar^{\f12}e^\Psi
\vf_\Phi\|_{\wt{L}^2_{t_0,t; \dot{\tht}(t)}(\cB^{1,0})}.
\end{split}
\eeq
As a result, it comes out
\beq \label{S4eq5}
\int_{t_0}^t\hbar(t')\bigl|\bigl(e^\Psi\D_k^{\rm h}\bigl[T^{\rm
h}_{\p_x \vf}u\bigr]_\Phi\ |\ e^\Psi\D_k^{\rm
h}\vf_\Phi\bigr)_{L^2_+}\bigr|\,dt'\lesssim  d_k^22^{-k}\|\hbar^{\f12}e^\Psi
\vf_\Phi\|_{\wt{L}^2_{t_0,t; \dot{\tht}(t)}(\cB^{1,0})}^2. \eeq

Finally again due to \eqref{S4eq3} and the support properties to the
Fourier transform of the terms in $R^{\rm h}(u,\p_x\vf),$ we get, by applying Lemma \ref{lem:Bern}, that
\beno
\begin{split}
&\int_{t_0}^t\hbar(t')\bigl|\bigl(e^\Psi\D_k^{\rm h}\bigl[R^{\rm h}(u,\p_x\vf
)\bigr]_\Phi\ |\ e^\Psi\D_k^{\rm
h}\vf_\Phi\bigr)_{L^2_+}\bigr|\,dt'\\
&\lesssim 2^{{\f{k}2}}\sum_{k'\geq
k-3}\int_{t_0}^t\|\wt{\D}^{\rm h}_{k'}u_\Phi(t')\|_{L^\infty_{\rm
v}(L^2_{\rm h})}\|\hbar^{\f12} e^{\Psi}\D_{k'}^{\rm
h}\p_x\vf_\Phi(t')\|_{L^2_+}\|\hbar^{\f12}e^{\Psi}\D_{k}^{\rm
h}\vf_\Phi(t')\|_{L^2_+}\,dt'\\
&\lesssim 2^{{\f{k}2}}\sum_{k'\geq
k-3}2^{\frac{k'}2}\int_{t_0}^t\dot{\theta}(t')\|\hbar^{\f12} e^{\Psi}\D_{k'}^{\rm
h}\vf_\Phi(t')\|_{L^2_+}\|\hbar^{\f12}e^{\Psi}\D_{k}^{\rm
h}\vf_\Phi(t')\|_{L^2_+}\,dt'\\
&\lesssim 2^{{\f{k}2}}\sum_{k'\geq
k-3}2^{\frac{k'}2}\Bigl(\int_{t_0}^t\dot{\theta}(t')\|\hbar^{\f12} e^{\Psi}\D_{k'}^{\rm
h}\vf_\Phi(t')\|_{L^2_+}^2\,dt'\Bigr)^{\frac12}\Bigl(\int_{t_0}^t\dot{\theta}(t')\|\hbar^{\f12}e^{\Psi}\D_{k}^{\rm
h}\vf_\Phi(t')\|_{L^2_+}^2\,dt'\Bigr)^{\frac12},
 \end{split} \eeno
 which together with Definition \ref{def1.1} ensures that
 \beq
 \label{S4eq6}
 \begin{split}
& \int_{t_0}^t\hbar(t')\bigl|\bigl(e^\Psi\D_k^{\rm h}\bigl[R^{\rm h}(u,\p_x\vf
)\bigr]_\Phi\ |\ e^\Psi\D_k^{\rm
h}\vf_\Phi\bigr)_{L^2_+}\bigr|\,dt'\\
&\lesssim d_k2^{-{\f{k}2}}\Bigl(\sum_{k'\geq
k-3}d_{k'}2^{-\frac{k'}2}\Bigr)\|\hbar^{\f12}e^\Psi
\vf_\Phi\|_{\wt{L}^2_{t_0,t;\dot{\tht}(t)}(\cB^{1,0})}^2\\
&\lesssim  d_k^22^{-k}\|\hbar^{\f12}e^\Psi
\vf_\Phi\|_{\wt{L}^2_{t_0,t;\dot{\tht}(t)}(\cB^{1,0})}^2.
\end{split}
\eeq

By summing up \eqref{S4eq4}, \eqref{S4eq5} and \eqref{S4eq6}, we achieve
\beq
\label{S4eq7}
\int_{t_0}^t\hbar(t')\bigl|\bigl(e^\Psi\D_k^{\rm
h}[u\p_x\vf]_\Phi\ |\ e^\Psi\D_k^{\rm h}\vf_\Phi\bigr)_{L^2_+}\bigr|\,dt'\lesssim  d_k^22^{-k}\|\hbar^{\f12}e^\Psi
\vf_\Phi\|_{\wt{L}^2_{t_0,t;\dot{\tht}(t)}(\cB^{1,0})}^2. \eeq

Whereas it is easy to observe that
\beq \label{S4eq9b}
\begin{split}
&\int_{t_0}^t\bigl(e^\Psi\D_k^{\rm h}[(u^s+\e f(t)\chi(y))\p_x \vf
]_\Phi\ |\ e^\Psi\D_k^{\rm
h}\vf_\Phi\bigr)_{L^2_+}\,dt'\\
&=\frac12\int_{t_0}^t\int_{\R_+} e^{2\Psi}(u^s+\e f(t)\chi(y))\int_{\R}\p_x\left(\D_k^{\rm
h}\vf_\Phi\right)^2\,dx\,dy\,dt'=0.
\end{split}
\eeq

Combining \eqref{S4eq7} with \eqref{S4eq9b} leads to \eqref{S4eq7a}. This finishes the proof of Lemma \ref{S4lem2}.
\end{proof}

\begin{proof}[Proof of Lemma \ref{S4lem3}]
By applying Bony's decomposition in the horizontal variable \eqref{Bony} to $\p_yu\p_x\vf,$ we write
\beno
\p_yu\p_x\vf=T^{\rm h}_{\p_yu}\p_x \vf+T^{\rm h}_{\p_x \vf}\p_yu+R^{\rm h}(\p_yu,\p_x \vf). \eeno
Considering
\eqref{1.8bb} and the support properties to the Fourier transform of
the terms in $T^{\rm h}_{\p_yu}\p_x \vf,$ we write
\beno
\begin{split}
&\int_{t_0}^t\hbar(t')\bigl|\bigl(e^\Psi\int_y^\infty \D_k^{\rm h}\bigl[T^{\rm
h}_{\p_yu} \p_x \vf\bigr]_\Phi\,dy'\ |\ e^\Psi\D_k^{\rm
h}\vf_\Phi\bigr)_{L^2_+}\bigr|\,dt'\\
&\lesssim \int_{t_0}^t\hbar(t')\bigl|\int_{\R^2_+} e^{-\f34\Psi}\int_y^\infty e^{\f74\Psi}\bigl|\D_k^{\rm h}\bigl[T^{\rm
h}_{\p_yu} \p_x \vf\bigr]_\Phi\bigr|\,dy'\ |\ e^\Psi|\D_k^{\rm
h}\vf_\Phi|\,dx\,dy\,dt'\\
&\lesssim \sum_{|k'-k|\leq
4}\int_{t_0}^t \|e^{-\f34\Psi(t')}\|_{L^2_{\rm v}}\|e^{\f34\Psi} S_{k'-1}^{\rm
h}\p_y u_\Phi(t')\|_{L^2_{\rm v}(L^\infty_\h)}\\
&\qquad\qquad\qquad\times \|\hbar^{\f12}e^{\Psi}\D_{k'}^{\rm h}\p_x
\vf_\Phi(t')\|_{L^2_+}\|\hbar^{\f12} e^{\Psi}\D_{k}^{\rm
h}\vf_\Phi(t')\|_{L^2_+}\,dt'\\
&\lesssim \sum_{|k'-k|\leq
4}2^{k'}\int_{t_0}^t \w{t'}^{\frac14}\|e^\Psi\p_yG_\Phi(t')\|_{\cB^{\frac12,0}} \|\hbar^{\f12}e^{\Psi}\D_{k'}^{\rm h}
\vf_\Phi(t')\|_{L^2_+}\|\hbar^{\f12} e^{\Psi}\D_{k}^{\rm
h}\vf_\Phi(t')\|_{L^2_+}\,dt',
\end{split} \eeno
where we used  Lemma \ref{lem:Bern}  and \eqref{S2eq20}  in the last step so that
\beno
\|e^{\f34\Psi} S_{k'-1}^{\rm
h}\p_y u_\Phi(t')\|_{L^2_{\rm v}(L^\infty_\h)}\lesssim \|e^\Psi\p_yG_\Phi(t')\|_{\cB^{\frac12,0}}.
\eeno
Then we get, by applying H\"older's inequality and \eqref{1.9}, that
\beno
\begin{split}
&\int_{t_0}^t\hbar(t')\bigl|\bigl(e^\Psi\int_y^\infty \D_k^{\rm h}\bigl[T^{\rm
h}_{\p_yu} \p_x \vf\bigr]_\Phi\,dy'\ |\ e^\Psi\D_k^{\rm
h}\vf_\Phi\bigr)_{L^2_+}\bigr|\,dt'\\
&\lesssim \sum_{|k'-k|\leq
4}2^{k'}\Bigl(\int_{t_0}^t \dot{\theta}(t') \|\hbar^{\f12}e^{\Psi}\D_{k'}^{\rm h}
\vf_\Phi(t')\|_{L^2_+}^2\,dt'\Bigr)^{\frac12}\Bigl(\int_{t_0}^t \dot{\theta}(t')\|\hbar^{\f12} e^{\Psi}\D_{k}^{\rm
h}\vf_\Phi(t')\|_{L^2_+}^2\,dt'\Bigr)^{\frac12}\\
&\lesssim  d_k^22^{-k}\|\hbar^{\f12}e^\Psi
\vf_\Phi\|_{\wt{L}^2_{t_0,t;\dot{\tht}(t)}(\cB^{1,0})}^2.
\end{split} \eeno

Along the same line, by virtue of \eqref{S2eq20},  we infer
\beno
\begin{split}
&\int_{t_0}^t\hbar(t')\bigl|\bigl(e^\Psi\int_y^\infty \D_k^{\rm h}\bigl[T^{\rm
h}_{ \p_x \vf}{\p_yu}\bigr]_\Phi\,dy'\ |\ e^\Psi\D_k^{\rm
h}\vf_\Phi\bigr)_{L^2_+}\bigr|\,dt'\\
&\lesssim \sum_{|k'-k|\leq
4}\int_{t_0}^t \w{t'}^{\frac14}\|\hbar^{\f12} e^\Psi S_{k'-1}^\h \p_x
\vf_\Phi(t')\|_{L^2_{\rm v}(L^\infty_\h)}  \|e^{\f34\Psi}\D_{k'}^{\rm h} \p_yu_\Phi(t')\|_{L^2_+}\|\hbar^{\f12} e^{\Psi}\D_{k}^{\rm
h}\vf_\Phi(t')\|_{L^2_+}\,dt'\\
&\lesssim \sum_{|k'-k|\leq
4}2^{-\frac{k'}2} \int_{t_0}^t \w{t'}^{\frac14}\|e^\Psi\p_yG_\Phi(t')\|_{\cB^{\frac12,0}}\|\hbar^{\f12} e^\Psi S_{k'-1}^\h \p_x
\vf_\Phi(t')\|_{L^2_{\rm v}(L^\infty_\h)}\|\hbar^{\f12} e^{\Psi}\D_{k}^{\rm
h}\vf_\Phi(t')\|_{L^2_+}\,dt'\\
&\lesssim \sum_{|k'-k|\leq
4}2^{-\frac{k'}2}\Bigl(\int_{t_0}^t \dot{\theta}(t')\|\hbar^{\f12} e^\Psi S_{k'-1}^\h \p_x
\vf_\Phi(t')\|_{L^2_{\rm v}(L^\infty_\h)}^2\,dt'\Bigr)^{\frac12}\\
&\qquad\qquad\qquad\qquad\qquad\qquad\qquad\times\Bigl(\int_{t_0}^t \dot{\theta}(t')\|\hbar^{\f12} e^{\Psi}\D_{k}^{\rm
h}\vf_\Phi(t')\|_{L^2_+}^2\,dt'\Bigr)^{\frac12},
\end{split} \eeno
from which, and \eqref{S4eq8}, we deduce that
\beno
\begin{split}
&\int_{t_0}^t\hbar(t')\bigl|\bigl(e^\Psi\int_y^\infty \D_k^{\rm h}\bigl[T^{\rm
h}_{ \p_x \vf}{\p_yu}\bigr]_\Phi\,dy'\ |\ e^\Psi\D_k^{\rm
h}\vf_\Phi\bigr)_{L^2_+}\bigr|\,dt'
\lesssim  d_k^22^{-k}\|\hbar^{\f12}e^\Psi
\vf_\Phi\|_{\wt{L}^2_{t_0,t;\dot{\tht}(t)}(\cB^{1,0})}^2.\end{split} \eeno

Finally due to  the support properties to the
Fourier transform of the terms in $R^{\rm h}(u,\p_x\vf),$ we get, by applying Lemma \ref{lem:Bern} and \eqref{S2eq20}, that
\beno
\begin{split}
&\int_{t_0}^t\hbar(t')\bigl|\bigl(e^\Psi\int_y^\infty \D_k^{\rm h}\bigl[R^{\rm h}(\p_yu,\p_x\vf
)\bigr]_\Phi\,dy'\ |\ e^\Psi\D_k^{\rm
h}\vf_\Phi\bigr)_{L^2_+}\bigr|\,dt'\\
&\lesssim 2^{{\f{k}2}}\sum_{k'\geq
k-3}\int_{t_0}^t \w{t'}^{\frac14}\|e^{\f34\Psi}\wt{\D}^{\rm h}_{k'}\p_y u_\Phi(t')\|_{L^2_+}\|\hbar^{\f12} e^{\Psi}\D_{k'}^{\rm
h}\p_x\vf_\Phi(t')\|_{L^2_+}\|\hbar^{\f12}e^{\Psi}\D_{k}^{\rm
h}\vf_\Phi(t')\|_{L^2_+}\,dt'\\
&\lesssim 2^{{\f{k}2}}\sum_{k'\geq
k-3}2^{\frac{k'}2}\int_{t_0}^t\dot{\theta}(t')\|\hbar^{\f12} e^{\Psi}\D_{k'}^{\rm
h}\vf_\Phi(t')\|_{L^2_+}\|\hbar^{\f12}e^{\Psi}\D_{k}^{\rm
h}\vf_\Phi(t')\|_{L^2_+}\,dt.
 \end{split} \eeno
Then a similar derivation of \eqref{S4eq6} leads to
 \beno
 \begin{split}
 \int_{t_0}^t\hbar(t')\bigl|\bigl(e^\Psi\int_y^\infty \D_k^{\rm h}\bigl[R^{\rm h}(\p_yu,\p_x\vf
)\bigr]_\Phi\,dy'\ |\ e^\Psi\D_k^{\rm
h}\vf_\Phi\bigr)_{L^2_+}\bigr|\,dt'
\lesssim  d_k^22^{-k}\|\hbar^{\f12}e^\Psi
\vf_\Phi\|_{\wt{L}^2_{t_0,t;\dot{\tht}(t)}(\cB^{1,0})}^2.
\end{split}
\eeno

As a consequence, we arrive at \eqref{S4eq9}.  This finishes the proof of Lemma \ref{S4lem3}. \end{proof}

\smallskip

\renewcommand{\theequation}{\thesection.\arabic{equation}}
\setcounter{equation}{0}

\section{Analytic energy estimate of $u$}\label{Sect5}

In this section,  we are going to present the weighted analytic energy estimate of $u$ and to obtain its decay-in-time estimate,
namely, we shall present the proof of Proposition \ref{S5prop2}.
The key ingredient will be the following proposition:

\begin{prop}\label{S5prop1}
{\sl
Let $\Phi(t,\xi)$ and $\Psi(t,y)$ be given by \eqref{eq2.6} and
\eqref{eq2.7} respectively. Let $u$ be a smooth enough solution of \eqref{S1eq5}. Then for any nonnegative and non-decreasing function $h\in C^1(\R_+),$ there exits a large enough constant $\lambda$ so that
 \beq \label{1.18}
\begin{split} \|\hbar^{\f12}e^\Psi
u_\Phi\|_{\wt{L}^\infty(t_0,t;\cB^{\f{1}2,0})}+&
\|\hbar^{\f12} e^\Psi
\p_y u_\Phi\|_{\wt{L}^2(t_0,t;\cB^{\f12,0})}\\
&\qquad\qquad\leq  \|\hbar^{\f12}e^\Psi
u_\Phi(t_0)\|_{\cB^{\f{1}2,0}}+\|\sqrt{\hbar'} e^\Psi u_\Phi\|_{\wt{L}^2(t_0,t;\cB^{\f12,0})}, \end{split} \eeq
 for any $t_0\in  [0,t]$
with $t<T^\ast,$ which is defined by \eqref{1.8a}.
  }
\end{prop}

\begin{proof}
 In view of
 \eqref{S1eq5}, we get, by a similar derivation of \eqref{S4eq2}, that
 \beq \label{1.8}
\begin{split}
\hbar(t)\bigl(&e^\Psi\D_k^{\rm h}\left(\p_t u_\Phi-\p_{yy}u_\Phi\right)\ |\ e^\Psi\D_k^{\rm
h}u_\Phi\bigr)_{L^2_+}+\lam\dot{\tht}(t)\hbar(t)\bigl(e^\Psi|D_h|\D_k^{\rm
h}u_\Phi\ |\ e^\Psi\D_k^{\rm
h}u_\Phi\bigr)_{L^2_+}\\
&+\hbar(t)\bigl(e^\Psi\D_k^{\rm h}[\left(u+u^s+\e f(t)\chi(y)\right)\p_x u
]_\Phi\ |\ e^\Psi\D_k^{\rm
h}u_\Phi\bigr)_{L^2_+}\\
&+\hbar(t)\bigl(e^\Psi\D_k^{\rm h}\bigl[v\p_y\left(u+u^s+\e f(t)\chi(y)\right)\bigr]_\Phi\ |\ e^\Psi\D_k^{\rm
h}u_\Phi\bigr)_{L^2_+}=0.
\end{split}
\eeq
 In what follows, we shall always assume that $t<T^\ast$ with
$T^\ast$ being determined by \eqref{1.8a} so that by
virtue of \eqref{eq2.6}, for any $t<T^\ast,$ there holds the
 convex inequality \eqref{1.8bb}.

Let us now handle term by term in \eqref{1.8}.\vspace{0.2cm}

Firstly due to $u|_{y=0}=0,$ we get, by a similar proof of Lemma \ref{S4lem1}, that

\beq\label{1.11}
\begin{split}
\int_{t_0}^t&\hbar(t')\bigl(e^\Psi\D_k^{\rm h}\left(\p_t u_\Phi-\p_{yy}u_\Phi\right) \ |\ e^\Psi\D_k^{\rm
h}u_\Phi\bigr)_{L^2_+}\,dt'\\
\geq &\f12\Bigl(\|\hbar^{\f12}e^\Psi\D_k^{\rm
h}u_\Phi(t)\|_{L^2_+}^2-\|\hbar^{\f12}e^{\Psi}\D_k^{\rm
h} u_\Phi(t_0)\|_{L^2_+}^2\\
&-\int_{t_0}^t\hbar'(t')\|e^\Psi\D_k^{\rm
h}u_\Phi(t')\|_{L^2_+}^2\,dt'+\|\hbar^{\f12} e^\Psi\D_k^{\rm
h}\p_yu_\Phi\|_{L^2_t(L^2_+)}^2\Bigr).
\end{split}
\eeq

While it follows  from Lemma \ref{S4lem2}  that
 \beq\label{1.13}
 \begin{split}
\int_{t_0}^t\hbar\bigl|\bigl(e^\Psi\D_k^{\rm h}[\left(u+u^s+\e f(t)\chi(y)\right)
\p_xu]_\Phi\ |\ &
e^\Psi\D_k^{\rm h}u_\Phi\bigr)_{L^2_+}\bigr|\,dt'\\
&\lesssim
d_k^22^{-k}\|\hbar^{\f12}e^\Psi
u_\Phi\|_{\wt{L}^2_{t_0,t;\dot{\tht}(t)}(\cB^{1,0})}^2. \end{split} \eeq

To deal with the estimate of $\int_{t_0}^t\hbar(t')\bigl(e^\Psi\D_k^{\rm
h}\bigl[v\p_yu\bigr]_\Phi\ |\ e^\Psi\D_k^{\rm
h}u_\Phi\bigr)_{L^2_+}\,dt',$ we need the following lemma, the proof of which will be postponed
at the end of this section.

\begin{lem}\label{S5lem1}
{\sl Under the assumptions of Proposition \ref{S5prop1}, for any $t_0\in [0,t]$ with $t<T^\ast,$ we have
\beq\label{1.14}
\int_{t_0}^t\hbar(t')\bigl|\bigl(e^\Psi\D_k^{\rm h}[v\p_yu]_\Phi\ |\ e^\Psi\D_k^{\rm
h}u_\Phi\bigr)_{L^2_+}\bigr|\,dt'\lesssim d_{k}^22^{-k}\|\hbar^{\f12}e^\Psi
u_\Phi\|_{\wt{L}^2_{t_0,t;\dot{\tht}(t)}(\cB^{1,0})}^2. \eeq}
\end{lem}

On the other hand,  due to $\p_xu+\p_yv=0,$ we have $v=-\int_y^\infty \p_yv\,dy'=\int_y^\infty \p_xu\,dy',$ so
 that it follows from Lemma \ref{lem:Bern} that for any $\ga\in (0,1)$
  \beq \label{ZPq}
\begin{split}
\bigl\|e^{\ga\Psi}{\D}_{k}^{\rm h}v_\Phi(t')\bigr\|_{L^\infty_{\rm v}(L^2_{\rm h})}\lesssim  &
2^k\bigl\|e^{\ga\Psi}\Bigl(\int_y^\infty e^{-2\Psi}\,dy'\Bigr)^{\f12}\Bigl(\int_0^\infty
|e^\Psi {\D}_{k}^{\rm h} u_\Phi(t')|^2\,dy\Bigr)^{\f12}\bigr\|_{L^\infty_{\rm v}(L^2_{\rm h})}\\
\lesssim & 2^{k'}\w{t'}^{\f14}\|e^{-\f{1-\ga}2\Psi}\|_{L^\infty_{\rm v}}\|e^{\Psi}{\D}_{k}^{\rm h}
u_\Phi(t')\|_{L^2_+}\\
\lesssim & 2^{k'}\w{t'}^{\f14}\|e^{\Psi}{\D}_{k}^{\rm h}
u_\Phi(t')\|_{L^2_+}.
\end{split}
\eeq
Then we deduce from  \eqref{S2eq13a} and \eqref{ZPq} that
\beno
\begin{split}
&\int_{t_0}^t\hbar(t')\bigl|\bigl(e^\Psi\D_k^{\rm h}v_\Phi
\p_y\left(u^s+\e f(t)\chi(y)\right)\ |\ e^\Psi\D_k^{\rm
h}u_\Phi\bigr)_{L^2_+}\bigr|\,dt'\\
&\leq \int_{t_0}^t\hbar(t')\bigl(\|e^{\f34\Psi}\p_yu^s\|_{L^2_{\rm v}}+\e f(t')\|e^\Psi\chi'\|_{L^2_{\rm v}}\bigr)
\bigl\|e^{\f\Psi4}\D_k^\h v_\Phi\|_{L^\infty_{\rm v}(L^2_\h)}
\|e^\Psi\D_k^\h u_\Phi\|_{L^2_+}\,dt'\\
&\lesssim 2^k\int_{t_0}^t\hbar(t')\w{t'}^{\f14}\bigl(\|e^\Psi\p_yG^s\|_{L^2_{\rm v}}+\e f(t')\|e^\Psi\chi'\|_{L^2_{\rm v}}\bigr)
\|e^\Psi\D_k^\h u_\Phi\|_{L^2_+}^2\,dt'.
\end{split}
\eeno
As a result, thanks to
\eqref{1.9} and Definition \ref{def1.1}, we achieve
\beq \label{4.6}
\begin{split}
\int_{t_0}^t\hbar(t')\bigl|\bigl(&e^\Psi\D_k^{\rm h}v_\Phi
\p_y\left(u^s+\e f(t)\chi(y)\right)\ |\ e^\Psi\D_k^{\rm
h}u_\Phi\bigr)_{L^2_+}\bigr|\,dt'\\
&\lesssim 2^k\int_{t_0}^t\dot\theta(t')
\|\hbar^{\f12}e^\Psi\D_k^\h u_\Phi\|_{L^2_+}^2\,dt'\lesssim d_k^22^{-k}\|\hbar^{\f12}e^\Psi
u_\Phi\|_{\wt{L}^2_{t_0,t;\dot{\tht}(t)}(\cB^{1,0})}^2.
\end{split}
\eeq

Whereas it follows from Lemma \ref{lem:Bern}  that
\beq \label{6.8} \lam\dot{\tht}(t)\bigl(e^\Psi|D_h|\D_k^{\rm h}u_\Phi\ |\
e^\Psi\D_k^{\rm h}u_\Phi\bigr)_{L^2_+}\geq
c\lam\dot{\tht}(t)2^k\|e^\Psi\D_{k}^{\rm h}u_\Phi(t)\|_{L^2_+}^2.
\eeq

By integrating \eqref{1.8} over
$[t_0,t]$ and then inserting the estimates, \eqref{1.11}, \eqref{1.13}, \eqref{1.14}, \eqref{4.6} and \eqref{6.8}, into the resulting
inequality, we conclude that
 \beq \label{4.7}
\begin{split} &\|\hbar^{\f12}e^\Psi\D_k^{\rm
h}u_\Phi\|_{L^\infty(t_0,t;L^2_+)}^2+2c\lam
2^k\int_{t_0}^t\dot{\tht}(t')\|e^\Psi\hbar^{\f12}\D_k^{\rm
h}u_\Phi(t')\|_{L^2_+}^2\,dt'\\
&\qquad+\|\hbar^{\f12}e^\Psi\D_k^{\rm
h}\p_yu_\Phi\|_{L^2(t_0,t;L^2_+)}^2\\
&\leq  \|\hbar^{\f12}e^{\Psi}\D_k^{\rm
h} u_\Phi(t_0)\|_{L^2_+}^2+\int_{t_0}^t\hbar'(t')\|e^\Psi\D_k^{\rm
h}u_\Phi(t')\|_{L^2_+}^2\,dt'\\
&\qquad\qquad\qquad\qquad\qquad\qquad+Cd_k^22^{-k} \|\hbar^{\f12} e^\Psi
u_\Phi\|_{\wt{L}^2_{t_0,t;\dot{\tht}(t)}(\cB^{1,0})}^2.\end{split}
\eeq
By taking square root of \eqref{4.7} and then multiplying
the resulting inequality by $2^{\f{k}2}$ and finally summing
over $k\in\Z,$ we find for any $t < T^\ast$ that \beq
\label{1.17}\begin{split} &\|\hbar^{\f12}e^\Psi
u_\Phi\|_{\wt{L}^\infty(t_0,t;\cB^{\f{1}2,0})}+\sqrt{2c\lam}\|\hbar^{\f12} e^\Psi
u_\Phi\|_{\wt{L}^2_{t_0,t;\dot{\tht}(t)}(\cB^{1,0})}+\|\hbar^{\f12} e^\Psi
\p_y u_\Phi\|_{\wt{L}^2(t_0,t;\cB^{\f12,0})}\\
&\qquad\leq \|\hbar^{\f12}e^\Psi
u_\Phi(t_0)\|_{\cB^{\f{1}2,0}}+\|\sqrt{\hbar'} e^\Psi
 u_\Phi\|_{\wt{L}^2(t_0,t;\cB^{\f12,0})}+\sqrt{C}\|\hbar^{\f12} e^\Psi
u_\Phi\|_{\wt{L}^2_{t_0,t;\dot{\tht}(t)}(\cB^{1,0})}. \end{split}
\eeq  By taking  $\lam$ in \eqref{1.17} to be a large enough positive constant so that
$c\lam\geq C,$  we deduce \eqref{1.18}. This completes the proof of Proposition \ref{S5prop1}.
\end{proof}

Now we are in a position to complete the proof of   Proposition \ref{S5prop2}.

\begin{proof}[Proof of Proposition \ref{S5prop2}]
Taking $\hbar(t)=1$ and $t_0=0$ in \eqref{1.18} gives rise to
 \beq \label{1.19}
\| e^\Psi
u_\Phi\|_{\wt{L}^\infty_t(\cB^{\f{1}2,0})}+\| e^\Psi
\p_y u_\Phi\|_{\wt{L}^2_t(\cB^{1,0})}\leq C\|e^{\f{y^2}8} e^{\delta|D_x|}u_0\|_{\cB^{\f{1}2,0}}. \eeq

While by taking $\hbar(t)=(t-t_0)$ and  $t_0=\f{t}2$ in \eqref{1.18}, we find
 \beno
\begin{split} \|t^{\f12}e^\Psi
u_\Phi(t)\|_{\cB^{\f{1}2,0}}\lesssim  \|(t'-t/2)^{\f12}e^\Psi
u_\Phi\|_{\wt{L}^\infty(t/2,t;\cB^{\f{1}2,0})}
\lesssim
\| e^\Psi
 u_\Phi\|_{\wt{L}^2(t/2,t;\cB^{\f12,0})}. \end{split} \eeno
Note that $u=\p_y\vf,$  by virtue of \eqref{S4eq11ag} and \eqref{S4eq16}, we achieve
\beq \label{S4eq11}
\begin{split} \|t^{\f12}e^\Psi
u_\Phi(t)\|_{\cB^{\f{1}2,0}}
\lesssim
 \|e^{\Psi}\vf_\Phi(t/2)\|_{\cB^{\f12,0}}\leq
C\|e^{\f{y^2}8} e^{\delta|D|}\vf_0\|_{\cB^{\f{1}2,0}}\w{t}^{-\f14},\end{split} \eeq

 Finally thanks to \eqref{S4eq11}, we get, by taking $\hbar(t)=t$ and then $t_0=\f{t}2$ in \eqref{1.18}, we obtain
\beno
\begin{split} \|(t')^{\f12}e^\Psi
u_\Phi\|_{\wt{L}^\infty(t/2,t;\cB^{\f{1}2,0})}
+&\|(t')^{\f12} e^\Psi
\p_y u_\Phi\|_{\wt{L}^2(t/2,t;\cB^{\f12,0})}\\
\leq & \|(t/2)^{\f12}e^\Psi
u_\Phi(t/2)\|_{\cB^{\f{1}2,0}}+
\| e^\Psi
 u_\Phi\|_{\wt{L}^2(t/2,t;\cB^{\f12,0})}\\
 \leq &
C\|e^{\f{y^2}8} e^{\delta|D|}\vf_0\|_{\cB^{\f{1}2,0}}\w{t}^{-\f14} \end{split} \eeno
 which together with \eqref{1.19} and \eqref{S4eq11} ensures \eqref{S4eq12}. This ends the proof of Proposition \ref{S5prop2}.
\end{proof}

Proposition \ref{S5prop1} has been proved provided that we present the proof of Lemma \ref{S5lem1}.

\begin{proof}[Proof of Lemma \ref{S5lem1}]
Once again we first get, by applying Bony's decomposition in the horizontal variable
\eqref{Bony} to $v\p_yu,$ that
\beq \label{pd1} v\p_yu=T^{\rm
h}_{v}\p_yu+T^{\rm h}_{\p_yu}v+R^{\rm h}(v,\p_yu).
 \eeq
 Considering
\eqref{1.8bb} and the support properties to the Fourier transform of
the terms in $T^{\rm h}_{v}\p_y u,$  and thanks to
  \eqref{S2eq20}, we get
   \beno \begin{split}
\int_{t_0}^t&\hbar(t')\bigl|\bigl(e^\Psi\D_k^{\rm h}\bigl[T^\h_{v}\p_yu\bigr]_\Phi\ |\ e^\Psi\D_k^{\rm
h}u_\Phi\bigr)_{L^2_+}\bigr|\,dt'\\
\lesssim &\sum_{|k'-k|\leq
4}\int_{t_0}^t\bigl\|\hbar^{\f12}e^{\f\Psi4} S_{k'-1}^{\rm h}v_\Phi(t')\bigr\|_{L^\infty_+}\|e^{\f34\Psi}\D_{k'}^{\rm
h}\p_yu_\Phi(t')\|_{L^2_+}\|\hbar^{\f12}e^{\Psi}\D_{k}^{\rm
h}u_\Phi(t')\|_{L^2_+}\,dt'\\
\lesssim &\sum_{|k'-k|\leq
4}2^{-\f{k'}2}\int_{t_0}^t\bigl\|\hbar^{\f12} e^{\f\Psi4} S_{k'-1}^{\rm
h}v_\Phi(t')\bigr\|_{L^\infty_+}\|e^{\Psi}\p_yG_\Phi(t')\|_{\cB^{\f{1}2,0}}\|\hbar^{\f12} e^{\Psi}\D_{k}^{\rm
h}u_\Phi(t')\|_{L^2_+}\,dt', \end{split} \eeno from which and
\eqref{1.9}, we infer
 \beno \begin{split}
\int_{t_0}^t&\hbar(t')\bigl|\bigl(e^\Psi\D_k^{\rm h}\bigl[T^\h_{v}\p_yu\bigr]_\Phi\ |\ e^\Psi\D_k^{\rm
h}u_\Phi\bigr)_{L^2_+}\bigr|\,dt'\\
\lesssim &\sum_{|k'-k|\leq
4}2^{-\f{k'}2}\Bigl(\int_{t_0}^t\w{t'}^{-\f12}\dot{\tht}(t')\bigl\|\hbar^{\f12}e^{\f\Psi4} S_{k'-1}^{\rm
h}
v_\Phi(t')\bigr\|_{L^\infty_+}^2\,dt'\Bigr)^{\f12}\\
&\qquad\qquad\qquad\qquad\qquad\qquad\qquad\qquad\quad\times\Bigl(\int_{t_0}^t\dot{\tht}(t')\|\hbar^{\f12}e^{\Psi}\D_{k}^{\rm
h}u_\Phi(t')\|_{L^2_+}^2\,dt'\Bigr)^{\f12}.
\end{split}
\eeno
 Whereas in view of Definition \ref{def1.1} and \eqref{ZPq}, we get,  by applying Lemma
\ref{lem:Bern}, that \beno
\begin{split}
\Bigl(\int_{t_0}^t\w{t'}^{-\f12}\dot{\tht}(t')\bigl\|\hbar^{\f12}e^{\f\Psi4} S_{k'-1}^{\rm
h}v\bigr\|_{L^\infty_+}^2\,dt'\Bigr)^{\f12}
\lesssim & \sum_{\ell\leq
k'-2}2^{\f32{\ell}}\Bigl(\int_{t_0}^t\dot{\tht}(t')\|\hbar^{\f12} e^{\Psi}\D_{\ell}^{\rm
h}u_\Phi(t')\|_{L^2_+}^2\,dt'\Bigr)^{\f12}\\
\lesssim & d_{k'}2^{\f{k'}2}\|\hbar^{\f12} e^\Psi
u_\Phi\|_{\wt{L}^2_{t_0,t;\dot{\tht}(t)}(\cB^{1,0})}.
\end{split}
\eeno
 Whence we obtain
  \beno \int_{t_0}^t\hbar(t')\bigl|\bigl(e^\Psi\D_k^{\rm
h}\bigl[T^\h_{v}\p_yu\bigr]_\Phi\ |\
e^\Psi\D_k^{\rm h}u_\Phi\bigr)_{L^2_+}\bigr|\,dt'\lesssim
d_{k}^22^{-k}\|\hbar^{\f12} e^\Psi
u_\Phi\|_{\wt{L}^2_{t_0,t;\dot{\tht}(t)}(\cB^{1,0})}^2. \eeno
By
the same manner, in view of \eqref{S2eq20} and \eqref{ZPq},  we infer
 \beno \begin{split}
\int_{t_0}^t&\hbar(t')\bigl|\bigl(e^\Psi\D_k^{\rm
h}\bigl[T^\h_{\p_yu}v\bigr]_\Phi\ |\
e^\Psi\D_k^{\rm h}u_\Phi\bigr)_{L^2_+}\bigr|\,dt'\\
\lesssim
&\sum_{|k'-k|\leq 4}\int_{t_0}^t\|e^{\f34\Psi} S_{k'-1}^{\rm
h}(\p_yu_\Phi(t'))\|_{L^2_{\rm v}(L^\infty_{\rm h})}\\
&\qquad\qquad\qquad\times \bigl\|e^{\f{\Psi}4}\hbar^{\f12}\D_{k'}^{\rm h}v_\Phi(t')\bigr\|_{L^\infty_{\rm v}(L^2_{\rm
h})}\|\hbar^{\f12} e^{\Psi}\D_{k}^{\rm
h}u_\Phi(t')\|_{L^2_+}\,dt'\\
\lesssim &\sum_{|k'-k|\leq
4}2^{k'}\int_{t_0}^t\|e^\Psi\p_yG_\Phi(t')\|_{\cB^{\f{1}2,0}}\w{t'}^{\f14}\|\hbar^{\f12}e^{\Psi}\D_{k'}^{\rm
h}u_\Phi(t')\|_{L^2}\|\hbar^{\f12} e^{\Psi}\D_{k}^{\rm
h}u_\Phi(t')\|_{L^2_+}\,dt',
\end{split}
\eeno from which, we get, by   a similar derivation of \eqref{S4eq4}, that
\beno \int_{t_0}^t\hbar(t')\bigl|\bigl(e^\Psi\D_k^{\rm
h}\bigl[T^\h_{\p_yu}v\bigr]_\Phi\ |\
e^\Psi\D_k^{\rm h}u_\Phi\bigr)_{L^2_+}\bigr|\,dt'\lesssim
d_{k}^22^{-k}\|\hbar^{\f12} e^\Psi
u_\Phi\|_{\wt{L}^2_{t_0,t;\dot{\tht}(t)}(\cB^{1,0})}^2. \eeno

Finally, considering the support properties to the
Fourier transform of the terms in $R^{\rm h}(v,\p_y u),$ we deduce from Lemma \ref{lem:Bern}  that
\beno\begin{split} \int_{t_0}^t&\hbar\bigl|\bigl(e^\Psi\D_k^{\rm
h}\bigl[R^{\rm h}(v,{\p_yu})\bigr]_\Phi\ |\ e^\Psi\D_k^{\rm
h}u_\Phi\bigr)_{L^2_+}\bigr|\,dt'\\
\lesssim
&2^{\f{k}2}\sum_{k'\geq k-3}\int_{t_0}^t\bigl\|\hbar^{\f12}e^{\f\Psi4} {\D}_{k'}^{\rm
h}v_\Phi(t')\bigr\|_{L^\infty_{\rm v}(L^2_{\rm
h})} \| e^{\f34\Psi}\wt{\D}_{k'}^{\rm
h}\p_yu_\Phi(t')\|_{L^2_+}\|\hbar^{\f12} e^{\Psi}\D_{k}^{\rm
h}u_\Phi(t')\|_{L^2_+}\,dt',
\end{split}
\eeno
from which, \eqref{S2eq20} and \eqref{ZPq},
we get, by a similar derivation of  \eqref{S4eq6}, that
\beno\begin{split}
\int_{t_0}^t&\hbar(t')\bigl|\bigl(e^\Psi\D_k^{\rm
h}\bigl[R^{\rm h}(v,{\p_yu})\bigr]_\Phi\ |\ e^\Psi\D_k^{\rm
h}u_\Phi\bigr)_{L^2_+}\bigr|\,dt'\\
\lesssim & 2^{\f{k}2}\sum_{k'\geq k-
3}2^{\f{k'}2}\int_{t_0}^t\w{t'}^{\f14}\|e^\Psi\p_yG_\Phi(t')\|_{\cB^{\f{1}2,0}}\|\hbar^{\f12} e^\Psi\D_{k'}^{\rm
h}u_\Phi(t')\|_{L^2_+}\|\hbar^{\f12} e^{\Psi}\D_{k}^{\rm
h}u_\Phi(t')\|_{L^2_+}\,dt'\\
\lesssim & 2^{\f{k}2}\sum_{k'\geq k-
3}2^{\f{k'}2}\Bigl(\int_{t_0}^t\dot{\theta}(t')\|\hbar^{\f12} e^\Psi\D_{k'}^{\rm
h}u_\Phi(t')\|_{L^2_+}^2\,dt'\Bigr)^{\f12}\Bigl(\int_{t_0}^t\dot{\theta}(t')\|\hbar^{\f12} e^{\Psi}\D_{k}^{\rm
h}u_\Phi(t')\|_{L^2_+}^2\,dt'\Bigr)^{\f12}\\
 \lesssim & d_{k}^22^{-k}\|\hbar^{\f12} e^\Psi
u_\Phi\|_{\wt{L}^2_{t_0,t;\dot{\tht}(t)}(\cB^{1,0})}^2.
\end{split}
\eeno

As a consequence, we achieve \eqref{1.14}. This finishes the proof of Lemma \ref{S5lem1}. \end{proof}

\smallskip

\renewcommand{\theequation}{\thesection.\arabic{equation}}
\setcounter{equation}{0}
\setcounter{equation}{0}
\section{The analytic energy estimate of the good quantity $G$}
\label{Sect7}

One key observation of this paper is that the weighted analytical norm of the function $g=\p_yG$ introduced in \eqref{S7eq4}
can control the evolution of the analytic radius to the solutions of \eqref{S1eq5}. In order to have a globally in time
estimate of the loss to the analytic radius of $u,$ we need the weighted analytical norm of  $\p_yG$  to decay fast enough as time
goes to $\infty.$
The goal of this section is to derive such a decay estimate of $G,$ namely, \eqref{S7eq24}.

Before preceding,  we first derive the equation satisfied by  $G,$ which is defined by \eqref{S7eq4}.
Indeed we observe from \eqref{S1eq2} that
\beq\label{S7eq1}
\begin{split}
&\p_t\bigl[\f{y}{2\w{t}}\vf\bigr]-\p_y^2\bigl[\f{y}{2\w{t}}\vf\bigr]+\w{t}^{-1}\bigl[u+\f{y}{2\w{t}}\vf\bigr]+\left(u+u^s+\e f(t)\chi(y)\right)\p_x\bigl[\f{y}{2\w{t}}\vf\bigr]\\
&\qquad\qquad\qquad\qquad\qquad\qquad\qquad+\f{y}{\w{t}}\int_y^\infty\left(\p_y \left(u+u^s+\e f(t)\chi(y')\right)\p_x\vf\right)\,dy'=0.
\end{split}
\eeq
Then by summing  up the $u$ equation of \eqref{S1eq5} with  \eqref{S7eq1}, we obtain the $G$ equation of \eqref{S7eq6}.
Moreover, due to $u|_{y=0}=0,$ we find $G|_{y=0}=0.$ As a consequence, $G$ verifies \eqref{S7eq6}.

The key ingredient used in the proof Proposition \ref{S7prop2} lies in  the following proposition:

\begin{prop}\label{S7prop1}
{\sl
Let $\Phi(t,\xi)$ and $\Psi(t,y)$ be given by \eqref{eq2.6} and
\eqref{eq2.7} respectively. Let the function $G$ be defined by \eqref{S7eq4}. Then for any nonnegative
and non-decreasing function $h\in C^1(\R_+),$ there exits a large enough constant $\lambda$ so that
 \beq \label{S7eq17}
\begin{split}
\f12&\|\hbar^{\f12}e^\Psi
\D_k^h G_\Phi\|_{{L}^\infty_{t}(L^2)}^2+\|\w{t'}^{-\f12}\hbar^{\f12}e^\Psi
\D_k^h G_\Phi\|_{{L}^2_{t}(L^2)}^2+\f12\|\hbar^{\f12}e^\Psi
\D_k^h \p_yG_\Phi\|_{{L}^2_{t}(L^2)}^2\\
&+c\lam 2^k\int_0^t\dot{\theta}(t')\|\hbar^{\f12}e^\Psi \D_k^h
G_\Phi(t')\|_{L^2}^2\,dt'\\
\leq & \f12\|\hbar^{\f12}e^\Psi
\D_k^h G_\Phi(0)\|_{L^2}^2+\f12\|\sqrt{\hbar'}e^\Psi
\D_k^h G_\Phi\|_{{L}^2_{t}(L^2)}^2+Cd_{k}^2
2^{-k}\|\hbar^{\f12}e^\Psi
G_\Phi\|_{\wt{L}^2_{t,\dot{\tht}(t)}(\cB^{1,0})}^2,
\end{split}
\eeq for any $t<T^\ast,$
 which is defined by \eqref{1.8a}.
}
\end{prop}

\begin{proof} In view of \eqref{S7eq6}, we get, by a similar derivation of \eqref{S4eq2}, that\beq \label{S7eq7}
\begin{split}
&\hbar(t)\bigl(e^\Psi\D_k^{\rm h}\left(\p_tG_\Phi-\p_{yy}G_\Phi+\w{t}^{-1}G_\Phi\right)\ |\ e^\Psi\D_k^{\rm
h}G_\Phi\bigr)_{L^2_+}\\
&+\lam\dot{\tht}(t)\hbar(t)\bigl(e^\Psi|D_h|\D_k^{\rm
h}G_\Phi\ |\ e^\Psi\D_k^\h  G_\Phi\bigr)_{L^2_+}+\hbar(t)\bigl(e^\Psi|D_h|\D_k^{\rm
h}[v\p_yG]_\Phi\ |\ e^\Psi\D_k^\h  G_\Phi\bigr)_{L^2_+}\\
&+\hbar(t)\bigl(e^\Psi  \D_k^{\rm h}\bigl[\left(u+u^s+\e f(t)\chi(y)\right)\p_xG\bigr]_\Phi
 \ |\ e^\Psi\D_k^{\rm
h} G_\Phi\bigr)_{L^2_+}\\
&+\hbar(t)\bigl(e^\Psi\D_k^{\rm h}\bigl[\p_y\left(u^s+\e f(t)\chi(y)\right)v -\f12\w{t}^{-1}v\p_y({y\vf})
\bigr]_\Phi\ |\ e^\Psi\D_k^{\rm
h}G_\Phi\bigr)_{L^2_+}\\
&+\w{t}^{-1}\hbar(t)\bigl(e^\Psi y\int_y^\infty \D_k^{\rm h}\bigl[\p_y\left(u+u^s+\e f(t)\chi(y')\right)\p_x\vf\bigr]_\Phi
\,dy' \ |\ e^\Psi\D_k^{\rm
h} G_\Phi\bigr)_{L^2_+}
=0.
\end{split}
\eeq
 In what follows, we shall always assume that $t<T^\ast$ with
$T^\ast$ being determined by \eqref{1.8a} so that by
virtue of \eqref{eq2.6}, for any $t<T^\ast,$ there holds the
 convex inequality \eqref{1.8bb}.

 \smallskip

Next let us handle term by term in \eqref{S7eq7}. \vspace{0.2cm}

Due to $G|_{y=0}=0,$ it follows from a similar proof of Lemma \ref{S4lem1}
that
\beq
\label{S7eq7ad}
\begin{split}
\int_{0}^t&\hbar(t')\bigl(e^\Psi\D_k^{\rm h}\left(\p_tG_\Phi-\p_{yy}G_\Phi\right)\ |\ e^\Psi\D_k^{\rm
h}G_\Phi\bigr)_{L^2_+}\,dt'\\
\geq & \f12\Bigl(\|\hbar^{\f12}e^\Psi\D_k^{\rm
h}G_\Phi(t)\|_{L^2_+}^2-\|\hbar^{\f12}e^{\Psi}\D_k^{\rm
h}G_\Phi(0)\|_{L^2_+}^2\\
&\quad-\int_{0}^t\hbar'(t')\|e^\Psi\D_k^{\rm
h}G_\Phi(t')\|_{L^2_+}^2\,dt'+\|e^\Psi\D_k^{\rm
h}\p_yG_\Phi\|_{L^2_t(L^2_+)}^2\Bigr).
  \end{split} \eeq

Whereas by applying Lemma \ref{S4lem2}, we find
\beq
\label{S7eq10}
\begin{split}
\int_{0}^t\hbar(t')\bigl|\bigl(e^\Psi\D_k^{\rm
h}[\left(u+u^s+\e f(t)\chi(y)\right)\p_xG]_\Phi\ |&\ e^\Psi\D_k^{\rm h}G_\Phi\bigr)_{L^2_+}\bigr|\,dt'\\
&\lesssim  d_k^22^{-k}\|\hbar^{\f12}e^\Psi
G_\Phi\|_{\wt{L}^2_{t,\dot{\tht}(t)}(\cB^{1,0})}^2. \end{split} \eeq

On the other hand, we observe from  the proof of\eqref{ZPq} that
\beno
\bigl\|e^{\f\Psi2}{\D}_{k}^{\rm h}v_\Phi(t')\bigr\|_{L^\infty_{\rm v}(L^2_{\rm h})}\lesssim
 2^{k'}\w{t'}^{\f14}\|e^{\f78\Psi}{\D}_{k}^{\rm h}
u_\Phi(t')\|_{L^2_+},
\eeno
which together with \eqref{S2eq21} implies that
\beq \label{S7eq11}
\bigl\|e^{\f\Psi2}{\D}_{k}^{\rm h}v_\Phi(t')\bigr\|_{L^\infty_{\rm v}(L^2_{\rm h})}
\lesssim \w{t}^{\f14}2^k\|e^\Psi \D_k^\h G_\Phi(t)\|_{L^2_+}.
\eeq
As a result, it comes out
\beno
\begin{split}
\int_0^t&\hbar(t')\bigl|\bigl(e^\Psi\p_y\left(u^s+\e f(t)\chi(y)\right)\D_k^{\rm h}v_\Phi\ |\ e^\Psi\D_k^{\rm
h}G_\Phi\bigr)_{L^2_+}\bigr|\,dt'\\
\lesssim &\int_0^t\hbar(t')\|e^{\f34\Psi}\p_y\left(u^s+\e f(t)\chi(y)\right)\|_{L^2_{\rm v}}\|e^{\f\Psi4}\D_k^{\rm
h} v_\Phi(t')\|_{L^\infty_{\rm v}(L^2_\h)}\|e^\Psi\D_k^{\rm h} G_\Phi(t')\|_{L^2_+}\,dt'\\
\lesssim &2^{k}\int_0^t\hbar(t')\w{t'}^{\f14}\|e^{\f34\Psi}\p_y\left(u^s+\e f(t)\chi(y)\right)\|_{L^2_{\rm v}}\|e^{\Psi}\D_k^{\rm
h} G_\Phi(t')\|_{L^2_+}^2\,dt'.
\end{split}
\eeno
This together with \eqref{1.9}, \eqref{S2eq13a} and Definition \ref{def1.1} ensures that
\beq \label{S7eq14po}
\begin{split}
\int_0^t&\hbar(t')\bigl|\bigl(e^\Psi\p_y\left(u^s+\e f(t)\chi(y)\right)\D_k^{\rm h}v_\Phi\ |\ e^\Psi\D_k^{\rm
h}G_\Phi\bigr)_{L^2_+}\bigr|\,dt'\\
\lesssim &2^k\int_0^t\dot{\theta}(t')\|\sqrt{\hbar}e^\Psi\D_k^{\rm
h} G_\Phi(t')\|_{L^2_+}^2\,dt'\\
\lesssim & d_k^22^{-{k}}\|\hbar^{\f12}e^\Psi
G_\Phi\|_{\wt{L}^2_{t,\dot{\tht}(t)}(\cB^{1,0})}^2.
\end{split}
\eeq

The estimate of the remaining terms in \eqref{S7eq7} relies on the following lemmas:

\begin{lem}\label{S7lem1}
{\sl For any $t<T^\ast,$ there holds
\beq \label{S7eq12}
\int_{0}^t\hbar(t')\bigl|\bigl(e^\Psi\D_k^{\rm h}\bigl[v \p_y G\bigr]_\Phi\ |\ e^\Psi\D_k^{\rm
h}G_\Phi\bigr)_{L^2_+}\bigr|\,dt'\lesssim  d_k^22^{-{k}}\|\hbar^{\f12}e^\Psi
G_\Phi\|_{\wt{L}^2_{t,\dot{\tht}(t)}(\cB^{1,0})}^2. \eeq}
\end{lem}

\begin{lem}\label{S7lem3}
{\sl For any $t<T^\ast,$ there holds
\beq \label{S7eq14}
\int_0^t\hbar(t')\w{t'}^{-1}\bigl|\bigl(e^\Psi \D_k^{\rm h}[v\p_y({y\vf})]_\Phi
 \ |\ e^\Psi\D_k^{\rm
h} G_\Phi\bigr)_{L^2_+}\bigr|\,dt'\lesssim  d_k^22^{-{k}}\|\hbar^{\f12}e^\Psi
G_\Phi\|_{\wt{L}^2_{t,\dot{\tht}(t)}(\cB^{1,0})}^2. \eeq
}
\end{lem}

\begin{lem}\label{S7lem2}
{\sl For any $t<T^\ast,$ there holds
\beq \label{S7eq13a}
\begin{split}
\int_{0}^t\w{t'}^{-1}\hbar(t')\bigl|\bigl(e^\Psi y\D_k^{\rm h}\bigl[\int_y^\infty \D_k^{\rm h}[\p_y&\left(u+u^s+\e f(t')\chi(y')\right)v]_\Phi
\,dy' \ |\ e^\Psi\D_k^{\rm
h}G_\Phi\bigr)_{L^2_+}\bigr|\,dt'\\
&\quad\qquad\qquad\qquad\quad\lesssim  d_k^22^{-{k}}\|\hbar^{\f12}e^\Psi
G_\Phi\|_{\wt{L}^2_{t,\dot{\tht}(t)}(\cB^{1,0})}^2. \end{split} \eeq}
\end{lem}

The proof of the above lemmas involves tedious calculations,  which we shall postpone
in the Appendix \ref{app}.

Now we admit  Lemmas \ref{S7lem1}-\ref{S7lem2} for the time being and continue the proof of Proposition \ref{S7prop1}.

As a matter of fact,
by integrating \eqref{S7eq7} over $[0,t]$ and then inserting the estimates (\ref{S7eq7ad}-\ref{S7eq13a})  into the resulting inequality,
 we achieve \eqref{S7eq17}. This completes the proof of Proposition \ref{S7prop1}.
\end{proof}

Now we  present the proof of Proposition \ref{S7prop2}.

\begin{proof}[Proof of Proposition \ref{S7prop2}]
It follows from  Lemma \ref{lem2.1} that
\beno
\f12\|\w{t'}^{-\f12}\hbar^{\f12}e^\Psi
\D_k^h G_\Phi\|_{{L}^2_{t}(L^2)}^2\leq \|\hbar^{\f12}e^\Psi
\D_k^h \p_yG_\Phi\|_{{L}^2_{t}(L^2)}^2.
\eeno
Inserting the above inequality into \eqref{S7eq17} and taking $\hbar(t)=\w{t}^{\f52}$ in the resulting inequality,
we find for any $t<T^\ast$
\beno
\begin{split}
2^k\|\w{t'}^{\f54}e^\Psi \D_k^\h
G_\Phi&\|_{{L}^\infty_{t}(L^2)}^2+c\lam 2^{2k}\int_0^t\dot{\theta}(t')\|\w{t'}^{\f54}e^\Psi
\D_k^\h G_\Phi(t')\|_{L^2}^2\,dt'\\
&\leq  2^k\|e^\Psi
\D_k^\h G_\Phi(0)\|_{L^2}^2+Cd_{k}^2
\|\w{t'}^{\f54}e^\Psi \D_k^\h
G_\Phi\|_{\wt{L}^2_{t,\dot{\tht}(t)}(\cB^{1,0})}^2.
\end{split}
\eeno
Taking square root of the above inequalities and then summing up the resulting ones gives rise
\beq \label{S7eq15be}
\begin{split}
\|\w{t'}^{\f54}e^\Psi
G_\Phi&\|_{\wt{L}^\infty_{t}(B^{\f12,0})}+\sqrt{c\lam}\|\w{t'}^{\f54} e^\Psi
G_\Phi\|_{\wt{L}^2_{t,\dot{\tht}(t)}(\cB^{1,0})}\\
&\leq  \|e^{\f{y^2}8}e^{\delta |D_\h|}
G_0\|_{\cB^{\f12,0}}+C\|\w{t'}^{\f54} e^\Psi
G_\Phi\|_{\wt{L}^2_{t,\dot{\tht}(t)}(\cB^{1,0})}.
\end{split}
\eeq
In particular, taking $\lam$ in \eqref{S7eq15be} so large that $c\lam\geq C,$ we achieve
\beq \label{S7eq15}
\|\w{t'}^{\f54}e^\Psi
G_\Phi\|_{\wt{L}^\infty_{t}(B^{\f12,0})} \leq  \|e^{\f{y^2}8}e^{\delta |D_\h|}
G_0\|_{\cB^{\f12,0}}\quad\mbox{ for any}\quad t<T^\ast.
\eeq

While by taking $\hbar(t)=1$ in \eqref{S7eq17}, we get, by a similar derivation of \eqref{S7eq15be}, that
\beno
\begin{split}
\sqrt{c\lam}\| e^\Psi
G_\Phi\|_{\wt{L}^2_{t,\dot{\tht}(t)}(\cB^{1,0})}+\|e^\Psi
\p_y G_\Phi\|_{\wt{L}^2_{t}(B^{\f12,0})}
\leq  \|e^{\f{y^2}8}e^{\delta |D_\h|}
G_0\|_{\cB^{\f12,0}}+C\| e^\Psi
G_\Phi\|_{\wt{L}^2_{t,\dot{\tht}(t)}(\cB^{1,0})}.
\end{split}
\eeno
By taking $\lam$  so large that $c\lam\geq C$ in the above inequality,  we obtain
\beq \label{S7eq15af}
\|e^\Psi
\p_y G_\Phi\|_{\wt{L}^2_{t}(B^{\f12,0})} \leq  \|e^{\f{y^2}8}e^{\delta |D_\h|}
G_0\|_{\cB^{\f12,0}} \quad\mbox{ for any}\quad t<T^\ast.
\eeq

 On the other hand, exactly along the same line to the proof of \eqref{S7eq17}, for any $t\in (0,T^\ast),$ we can show that
\beno
\begin{split}
\f12&\|\w{t'}^{\f54}e^\Psi \D_k^\h
G_\Phi(t)\|_{L^2}^2+\f12\int_{\f{t}2}^t\|\w{t'}^{\f54}e^\Psi \D_k^\h
\p_yG_\Phi\|_{L^2}^2\,dt'\\
&+c\lam 2^k\int_{\f{t}2}^t\dot{\theta}(t')\|\w{t'}^{\f54}e^\Psi \D_k^\h
G_\Phi(t')\|_{L^2}^2\,dt'\\
\leq & \f12\|\w{t/2}^{\f54}e^\Psi \D_k^\h
G_\Phi(t/2)\|_{L^2}^2+\f14\int_{\f{t}2}^t\w{t'}^{\f32}\|e^\Psi \D_k^\h
G_\Phi(t')\|_{L^2}^2\,dt'\\
&\qquad\qquad\qquad\qquad\qquad\qquad+Cd_{k}^2
2^{-k}\|\w{t'}^{\f54}e^\Psi
G_\Phi\|_{\wt{L}^2_{t/2,t;\dot{\tht}(t)}(\cB^{1,0})}^2,
\end{split}
\eeno
from which and \eqref{S7eq15}, we get, by a similar derivation of \eqref{S7eq15be}, that
\beno
\begin{split}
\|\w{t}^{\f54}&e^\Psi
G_\Phi(t)\|_{B^{\f12,0}}+\sqrt{2c\lam}\|\w{t'}^{\f54} e^\Psi
G_\Phi\|_{\wt{L}^2_{t/2,t;\dot{\tht}(t)}(\cB^{1,0})}+\|\w{t'}^{\f54} e^\Psi \p_yG_\Phi\|_{\wt{L}^2(t/2,t;\cB^{\f12,0})}\\
&\leq  \|e^{\f{y^2}8}e^{\delta |D_\h|}
G(0)\|_{\cB^{\f12,0}}+\|\w{t'}^{\f34}e^\Psi G_\Phi\|_{\wt{L}^2(t/2,t;\cB^{\f12,0})}+\sqrt{C}\|\w{t'}^{\f54} e^\Psi
G_\Phi\|_{\wt{L}^2_{t/2,t;\dot{\tht}(t)}(\cB^{1,0})}.
\end{split}
\eeno
Yet it follows from \eqref{S7eq15} that for any $t\in (0,T^\ast)$
\beno
\|\w{t'}^{\f34}e^\Psi G_\Phi\|_{\wt{L}^2(t/2,t;\cB^{\f12,0})}\lesssim \|\w{t'}^{\f54}e^\Psi
G_\Phi\|_{\wt{L}^\infty_{t}(B^{\f12,0})}\lesssim \|e^{\f{y^2}8}e^{\delta |D_\h|}
G_0\|_{\cB^{\f12,0}}.
\eeno
As a consequence, as long as $c\lam\geq C,$ we arrive at
\beq \label{S7eq18}
\|\w{t'}^{\f54}\p_yG_\Phi\|_{\wt{L}^2(t/2,t;\cB^{\f12,0})}\lesssim \|e^{\f{y^2}8}e^{\delta |D_\h|}
G_0\|_{\cB^{\f12,0}} \quad\mbox{ for any}\quad t<T^\ast. \eeq

With \eqref{S7eq15} and \eqref{S7eq18}, to finish the proof of \eqref{S7eq24},
it remains to show  that for any $t<T^\ast,$
\beq \label{S5eq6}
\int_0^t\w{t'}^{\f14}\bigl\|e^\Psi\p_yG_\Phi(t')\bigr\|_{\cB^{\f12,0}}\,dt'\leq C\|e^{\f{y^2}8}e^{\delta |D_\h|}
G_0\|_{\cB^{\f12,0}}.
\eeq
Indeed for any $t<T^\ast$ and $t>1,$ there exists a unique integer $N_t$ so that $2^{N_t-1}<t\leq 2^{N_t}.$ Then
we have $\f{t}2\leq 2^{N_t-1},$ so that there holds
\beno
\begin{split}
\int_{2^{N_t-1}}^t\w{t'}^{\f14}\bigl\|e^\Psi\p_yG_\Phi(t')\bigr\|_{\cB^{\f12,0}}\,dt'
\leq &\Bigl( \int_{2^{N_t-1}}^t\w{t'}^{-2}\,dt'\Bigr)^{\f12}  \\ &\times \Bigl(\int_{{t}/2}^t\bigl(\w{t'}^{\f54}\bigl\|e^\Psi\p_yG_\Phi(t')\bigr\|_{\cB^{\f12,0}}\bigr)^2\,dt'\Bigr)^{\f12}\\
\leq & C2^{-\f{N_t}2}\|\w{t'}^{\f54}\p_yG_\Phi\|_{\wt{L}^2(t/2,t;\cB^{\f12,0})}.
\end{split}
\eeno

Along the same line for any $j\in [0, N_t-2],$ we find
\beno
\begin{split}
\int_{2^{j}}^{2^{j+1}}\w{t'}^{\f14}\bigl\|e^\Psi\p_yu_\Phi(t')\bigr\|_{\cB^{\f12,0}}\,dt'
\leq &\Bigl( \int^{2^{j+1}}_{2^j}\w{t'}^{-2}\,dt'\Bigr)^{\f12}  \\ &\times \Bigl(\int^{2^{j+1}}_{2^j}\bigl(\w{t'}^{\f54}\bigl\|e^\Psi\p_yu_\Phi(t')\bigr\|_{\cB^{\f12,0}}\bigr)^2\,dt'\Bigr)^{\f12}\\
\leq & C 2^{-\f{j}2} \|\w{t'}^{\f54}\p_yG_\Phi\|_{\wt{L}^2(2^j,2^{j+1};\cB^{\f12,0})}
\end{split}
\eeno
As a consequence,   we deduce from \eqref{S7eq15af}, \eqref{S7eq18} and the above inequalities that
\beno
\begin{split}
\int_0^t&\w{t'}^{\f14}\bigl\|e^\Psi\p_yG_\Phi(t')\bigr\|_{\cB^{\f12,0}}\,dt'\leq   \|e^\Psi
\p_y G_\Phi\|_{{L}^2(0,1; B^{\f12,0})}
\\
&+\sum_{j=0}^{N_t-2}\int_{2^{j}}^{2^{j+1}}\w{t'}^{\f14}\bigl\|e^\Psi\p_yG_\Phi(t')\bigr\|_{\cB^{\f12,0}}\,dt'+\int_{2^{N_t-1}}^t\w{t'}^{\f14}\bigl\|e^\Psi\p_yG_\Phi(t')\bigr\|_{\cB^{\f12,0}}\,dt'\\
\leq &C\|e^{\f{y^2}8}e^{\delta |D_\h|}
G_0\|_{\cB^{\f12,0}}\Bigl(1+\sum_{j=0}^\infty 2^{-\f{j}2}\Bigr)\leq C\|e^{\f{y^2}8}e^{\delta |D_\h|}
G_0\|_{\cB^{\f12,0}}.
\end{split}
\eeno
This leads to \eqref{S5eq6}. We thus complete the proof of Proposition \ref{S7prop2}.
\end{proof}

\appendix

\renewcommand{\theequation}{\thesection.\arabic{equation}}
\setcounter{equation}{0}
\section{The proof of Lemmas \ref{S7lem1}-\ref{S7lem2}}\label{app}

In this appendix, we shall present the proof of Lemmas \ref{S7lem1}-\ref{S7lem2}.

\begin{proof}[Proof of Lemma \ref{S7lem1}] We first get, by applying Bony's decomposition \eqref{Bony} in the horizontal
variable  to $v\p_y G,$  that
\beno v\p_yG=T^{\rm h}_v\p_y G+T^{\rm h}_{\p_yG }v+R^{\rm h}(v,\p_y G). \eeno
Considering
\eqref{1.8bb} and the support properties to the Fourier transform of
the terms in $T^{\rm h}_v\p_y G,$ we write
\beno
\begin{split}
\int_{0}^t&\hbar(t')\bigl|\bigl(e^\Psi\D_k^{\rm h}\bigl[T^{\rm
h}_v \p_y G\bigr]_\Phi\ |\ e^\Psi\D_k^{\rm
h}G_\Phi\bigr)_{L^2_+}\bigr|\,dt'\\
\lesssim & \sum_{|k'-k|\leq
4}\int_{0}^t\| \hbar^{\f12} S_{k'-1}^{\rm
h}v_\Phi(t')\|_{L^\infty_+}\|e^{\Psi}\D_{k'}^{\rm h}\p_y
G_\Phi(t')\|_{L^2_+}\|\hbar^{\f12} e^{\Psi}\D_{k}^{\rm
h}G_\Phi(t')\|_{L^2_+}\,dt'\\
\lesssim & \sum_{|k'-k|\leq
4} 2^{-\f{k'}2}\int_{0}^t\| \hbar^{\f12} S_{k'-1}^{\rm
h}v_\Phi(t')\|_{L^\infty_+}\|e^{\Psi}\p_y
G_\Phi(t')\|_{\cB^{\f12,0}}\|\hbar^{\f12} e^{\Psi}\D_{k}^{\rm
h}G_\Phi(t')\|_{L^2_+}\,dt'.
\end{split}
\eeno
Then in view of \eqref{1.9}, by applying H\"older's inequality, we find
\beno
\begin{split}
\int_{0}^t\hbar(t')&\bigl|\bigl(e^\Psi\D_k^{\rm h}\bigl[T^{\rm
h}_v \p_y G\bigr]_\Phi\ |\ e^\Psi\D_k^{\rm
h}G_\Phi\bigr)_{L^2_+}\bigr|\,dt'\\
\lesssim & \sum_{|k'-k|\leq
4} 2^{-\f{k'}2}\Bigl(\int_{0}^t\dot{\theta}(t')\w{t'}^{-\f12}\| \hbar^{\f12} S_{k'-1}^{\rm
h}v_\Phi(t')\|_{L^\infty_+}^2\,dt'\Bigr)^{\f12}\\
&\qquad\qquad\qquad\qquad\qquad\times\Bigl(\int_{0}^t\dot{\theta}(t')\|\hbar^{\f12} e^{\Psi}\D_{k}^{\rm
h}G_\Phi(t')\|_{L^2_+}^2\,dt'\Bigr)^{\f12}.
\end{split}
\eeno
Yet  in view of \eqref{S7eq11}, we get, by a similar derivation of  \eqref{S4eq8}, that
\beq \label{S7eq11n}
\Bigl(\int_{0}^t\dot{\theta}(t')\w{t'}^{-\f12}\| \hbar^{\f12} S_{k'-1}^{\rm
h}v_\Phi(t')\|_{L^\infty_+}^2\,dt'\Bigr)^{\f12}\lesssim  d_{k'}2^{\f{k'}2}\|\hbar^{\f12}e^\Psi
G_\Phi\|_{\wt{L}^2_{t,\dot{\tht}(t)}(\cB^{1,0})}.
\eeq
As a result, we deduce from Definition \ref{def1.1} that
\beno
\int_{0}^t&\hbar(t')\bigl|\bigl(e^\Psi\D_k^{\rm h}\bigl[T^{\rm
h}_v \p_y G\bigr]_\Phi\ |\ e^\Psi\D_k^{\rm
h}G_\Phi\bigr)_{L^2_+}\bigr|\,dt'\lesssim  d_k^22^{-{k}}\|\hbar^{\f12}e^\Psi
G_\Phi\|_{\wt{L}^2_{t,\dot{\tht}(t)}(\cB^{1,0})}^2.
\eeno

Along the same line, we de deduce from \eqref{1.9} and \eqref{S7eq11} that
\beno
\begin{split}
\int_{0}^t&\hbar(t')\bigl|\bigl(e^\Psi\D_k^{\rm h}\bigl[T^{\rm
h}_{\p_y G}v\bigr]_\Phi\ |\ e^\Psi\D_k^{\rm
h}G_\Phi\bigr)_{L^2_+}\bigr|\,dt'\\
\lesssim & \sum_{|k'-k|\leq
4}\int_{0}^t\| e^\Psi S_{k'-1}^{\rm
h}\p_yG_\Phi(t')\|_{L^2_{\rm v}(L^\infty_\h)}\|\hbar^{\f12} \D_{k'}^{\rm h}v_\Phi(t')\|_{L^\infty_{\rm v}(L^2_\h)}\|\hbar^{\f12} e^{\Psi}\D_{k}^{\rm
h}G_\Phi(t')\|_{L^2_+}\,dt'\\
\lesssim & \sum_{|k'-k|\leq 4}2^{k'}\int_{0}^t\dot\theta(t')\| \hbar^{\f12} e^\Psi\D_{k'}^{\rm h}G_\Phi(t')\|_{L^2_+}\|\hbar^{\f12} e^{\Psi}\D_{k}^{\rm
h}G_\Phi(t')\|_{L^2_+}\,dt',
\end{split}
\eeno
Then a similar derivation of \eqref{S4eq4} yields
\beno
\int_{0}^t\hbar(t')\bigl|\bigl(e^\Psi\D_k^{\rm h}\bigl[T^{\rm
h}_{\p_y G}v\bigr]_\Phi\ |\ e^\Psi\D_k^{\rm
h}G_\Phi\bigr)_{L^2_+}\bigr|\,dt'\lesssim  d_k^22^{-{k}}\|\hbar^{\f12}e^\Psi
G_\Phi\|_{\wt{L}^2_{t,\dot{\tht}(t)}(\cB^{1,0})}^2.
\eeno

Finally again due to \eqref{S7eq11} and the support properties to the
Fourier transform of the terms in $R^{\rm h}(v,\p_yG),$ we get, by applying Lemma \ref{lem:Bern}, that
\beno
\begin{split}
&\int_{0}^t\hbar(t')\bigl|\bigl(e^\Psi\D_k^{\rm h}\bigl[R^{\rm h}(v,\p_yG
)\bigr]_\Phi\ |\ e^\Psi\D_k^{\rm
h}G_\Phi\bigr)_{L^2_+}\bigr|\,dt'\\
&\lesssim 2^{{\f{k}2}}\sum_{k'\geq
k-3}\int_{0}^t\|\hbar^{\f12}{\D}^{\rm h}_{k'}v_\Phi(t')\|_{L^\infty_{\rm
v}(L^2_{\rm h})}\|e^{\Psi}\wt{\D}_{k'}^{\rm
h}\p_yG_\Phi(t')\|_{L^2_+}\|\hbar^{\f12}e^{\Psi}\D_{k}^{\rm
h}G_\Phi(t')\|_{L^2_+}\,dt'\\
&\lesssim 2^{{\f{k}2}}\sum_{k'\geq
k-3}2^{\frac{k'}2}\int_{0}^t\dot{\theta}(t')\|\hbar^{\f12} e^{\Psi}\D_{k'}^{\rm
h}G_\Phi(t')\|_{L^2_+}\|\hbar^{\f12}e^{\Psi}\D_{k}^{\rm
h}G_\Phi(t')\|_{L^2_+}\,dt',
 \end{split} \eeno
 from which, we get, by a similar derivation of \eqref{S4eq6}, that
 \beno
 \int_{0}^t\hbar(t')\bigl|\bigl(e^\Psi\D_k^{\rm h}\bigl[R^{\rm h}(v,\p_yG
)\bigr]_\Phi\ |\ e^\Psi\D_k^{\rm
h}G_\Phi\bigr)_{L^2_+}\bigr|\,dt'\lesssim  d_k^22^{-{k}}\|\hbar^{\f12}e^\Psi
G_\Phi\|_{\wt{L}^2_{t,\dot{\tht}(t)}(\cB^{1,0})}^2.
\eeno

Summing up the above estimates gives rise to \eqref{S7eq12}. This finishes the proof of Lemma \ref{S7lem1}.
\end{proof}

\begin{proof}[Proof of Lemma \ref{S7lem3}]
Applying Bony's decomposition \eqref{Bony} in the horizontal to $v\p_y({y\vf})$ yields
\beno
v\p_y({y\vf})=T^\h_v\p_y({y\vf})+T^\h_{\p_y({y\vf})}v+R^\h(v,\p_y({y\vf})).
\eeno
In view of \eqref{1.8bb} and \eqref{S7eq21}, we infer
\beno
\begin{split}
 \int_0^t&\hbar(t')\w{t'}^{-1}
\bigl|\bigl(e^\Psi\D_k^\h[T^\h_v\p_y({y\vf})]_\Phi\ |\ e^\Psi\D_k^{\rm
h} G_\Phi\bigr)_{L^2_+}\bigr|\,dt'\\
\lesssim & \sum_{|k'-k|\leq 4}\int_0^t\hbar(t')\w{t'}^{-1}\|e^{\f\Psi2}S_{k'-1}^\h v_\Phi(t')\|_{L^\infty_+}\|e^{\f\Psi2}\D_{k'}^\h \p_y({y\vf}_\Phi)(t')\|_{L^2_+}\|e^\Psi\D_k^\h G_\Phi(t')\|_{L^2_+}\,dt'\\
\lesssim & \sum_{|k'-k|\leq 4}\int_0^t\hbar(t')\|e^{\f\Psi2}S_{k'-1}^\h v_\Phi(t')\|_{L^\infty_+}
\|e^\Psi\D_{k'}^\h \p_yG_\Phi(t')\|_{L^2_+}
\|e^\Psi\D_k^\h G_\Phi(t')\|_{L^2_+}\,dt'\\
\lesssim & \sum_{|k'-k|\leq 4}2^{-\f{k'}2}\Bigl(\int_0^t\dot{\theta}(t')\w{t'}^{-\f12}\|\hbar^{\f12} e^{\f\Psi2} S_{k'-1}^\h v_\Phi(t')\|_{L^\infty_+}^2\,dt'\Bigr)^{\f12}\\
&\qquad\qquad\qquad\qquad\qquad\qquad\qquad\qquad\times\Bigl(\int_0^t\dot{\theta}
\|\hbar^{\f12} e^\Psi \D_k^\h G_\Phi(t')\|_{L^2_+}^2\,dt'\Bigr)^{\f12},
\end{split}
\eeno
which together with  Definition \ref{def1.1} and \eqref{S7eq11n} ensures that
\beno
\begin{split}
  \int_0^t&\hbar(t')\w{t'}^{-1}
\bigl|\bigl(e^\Psi\D_k^\h[T^\h_v\p_y({y\vf})]_\Phi\ | \ e^\Psi\D_k^{\rm
h} G_\Phi\bigr)_{L^2_+}\bigr|\,dt'
\lesssim d_k^22^{-k}\|\hbar^{\f12}e^\Psi
G_\Phi\|_{\wt{L}^2_{t,\dot{\tht}(t)}(\cB^{1,0})}^2.
\end{split}
\eeno

Similarly,  by virtue of  \eqref{S7eq21} and \eqref{S7eq11}, we have
\beno
\begin{split}
 \int_0^t&\hbar(t')\w{t'}^{-1}
\bigl|\bigl(e^\Psi\D_k^\h[T^\h_{\p_y({y\vf})}v]_\Phi\ |\ e^\Psi\D_k^{\rm
h} G_\Phi\bigr)_{L^2_+}\bigr|\,dt'\\
\lesssim & \sum_{|k'-k|\leq 4}\int_0^t\hbar(t')\w{t'}^{-1}\|e^{\f\Psi2}S_{k'-1}^\h \p_y({y\vf}_\Phi)(t')\|_{L^2_{\rm v}(L^\infty_\h)}\\
&\qquad\qquad\qquad\qquad\times\|e^{\f\Psi2}\D_{k'}^\h v_\Phi(t')\|_{L^\infty_{\rm v}(L^2_\h)}\|e^\Psi\D_k^\h G_\Phi(t')\|_{L^2_+}\,dt'\\
\lesssim & \sum_{|k'-k|\leq 4}2^{\f{k'}2}\int_0^t\dot\theta(t')
\|\hbar^{\f12} e^\Psi\D_{k'}^\h G_\Phi(t')\|_{L^2_+}
\|\hbar^{\f12} e^\Psi\D_k^\h G_\Phi(t')\|_{L^2_+}\,dt'.
\end{split}
\eeno
As a result, we deduce, by a similar derivation of \eqref{S4eq4}, that
\beno
 \int_0^t\hbar(t')\w{t'}^{-1}
\bigl|\bigl(e^\Psi\D_k^\h[T^\h_{\p_y({y\vf})}v]_\Phi\ |\ e^\Psi\D_k^{\rm
h} G_\Phi\bigr)_{L^2_+}\bigr|\,dt'\lesssim d_{k'}^2
2^{-k'}\|\hbar^{\f12}e^\Psi
G_\Phi\|_{\wt{L}^2_{t,\dot{\tht}(t)}(\cB^{1,0})}^2.
\eeno

Finally again thanks to \eqref{S7eq21} and \eqref{S7eq11} , we get, by applying Lemma \ref{lem:Bern}, that
\beno
\begin{split}
 \int_0^t&\hbar(t')\w{t'}^{-1}
\bigl|\bigl(e^\Psi\D_k^\h[R^\h(v,\p_y({y\vf}))]_\Phi\ |\ e^\Psi\D_k^{\rm
h} G_\Phi\bigr)_{L^2_+}\bigr|\,dt'\\
\lesssim & 2^{\f{k'}2}\sum_{k'\geq k-3}\int_0^t\hbar(t')\w{t'}^{-1}\|e^{\f\Psi2}{\D}_{k'}^\h v_\Phi(t')\|_{L^\infty_{\rm v}(L^2_\h)}\\
&\qquad\qquad\qquad\qquad\times\|e^{\f\Psi2}\wt{\D}_{k'}^\h \p_y({y\vf}_\Phi)(t')\|_{L^2_+}
\|e^\Psi\D_k^\h G_\Phi(t')\|_{L^2_+}\,dt'\\
\lesssim &  2^{\f{k'}2}\sum_{k'\geq k-3}2^{\f{k'}2}\int_0^t\hbar(t')\dot\theta(t')
\|e^{\Psi} \D_{k'}^\h G_\Phi(t')\|_{L^2}
\|e^\Psi\D_k^\h G_\Phi(t')\|_{L^2_+}\,dt'.
\end{split}
\eeno
Then it follows from a similar derivation of \eqref{S4eq6} that
\beno
 \int_0^t\hbar(t')\w{t'}^{-1}
\bigl|\bigl(e^\Psi\D_k^\h[R^\h(v,\p_y({y\vf}))]_\Phi\ |\ e^\Psi\D_k^{\rm
h} G_\Phi\bigr)_{L^2_+}\bigr|\,dt'\lesssim d_{k'}^2
2^{-k'}\|\hbar^{\f12}e^\Psi
G_\Phi\|_{\wt{L}^2_{t,\dot{\tht}(t)}(\cB^{1,0})}^2.
\eeno

By summarizing the above estimates, we conclude the proof of \eqref{S7eq14}. This ends the proof of Lemma \ref{S7lem3}.
\end{proof}

\begin{proof}[Proof of Lemma \ref{S7lem2}] We first observe from $\p_xu+\p_yv=0$ that $v=\int_y^\infty \p_xu\,dy',$ so that one has
\beno
|\D_k^\h v_\Phi(t)|\leq e^{-\f58\Psi}\int_y^\infty e^{-\f18\Psi} \times e^{\f34\Psi}|\D_k^\h\p_xu_\Phi(t)|\,dy',
\eeno
from which, \eqref{S2eq21} and Lemma \ref{lem:Bern}, we infer
\beq \label{AAAP}
\begin{split}
\|e^{\f\Psi2}\D_k^\h v_\Phi(t)\|_{L^2}\leq &\|e^{-\f{\Psi}8}\|_{L^2_{\rm v}}^2\|e^{\f34\Psi}\D_k^\h\p_xu_\Phi(t)\|_{L^2_+}\\
\lesssim &2^k\w{t}^{\f12}\|e^{\f34\Psi}\D_k^\h u_\Phi(t)\|_{L^2_+}\\
\lesssim &2^k\w{t}^{\f12}\|e^{\Psi}\D_k^\h G_\Phi(t)|\|_{L^2_+}.
\end{split}
\eeq

 In view of \eqref{1.9}, \eqref{S2eq13a} and \eqref{AAAP}, we infer
\beq \label{S7eq13b}
\begin{split}
&\int_0^t\w{t'}^{-1}\hbar(t')\bigl|\bigl(e^\Psi y\int_y^\infty \left(\p_yu^s+\e f(t')\chi'(y')\right)\D_k^{\rm h}v_\Phi
\,dy' \ |\ e^\Psi\D_k^{\rm
h} G_\Phi\bigr)_{L^2_+}\bigr|\,dt'\\
&\lesssim \int_0^t\hbar(t')\w{t'}^{-1}\|e^{-\f\Psi4}y\|_{L^2_{\rm v}}\|e^{\f34\Psi}\left(\p_yu^s+\e f(t')\chi'(y')\right)\|_{L^2_{\rm v}}
\|e^{\f\Psi2}\D_k^{\rm h}v_\Phi\|_{L^2_+}\|\D_k^{\rm h}G_\Phi\|_{L^2_+}\,dt'\\
&\lesssim \int_{0}^t \dot{\theta}(t')
\|\hbar^{\f12}e^\Psi\D_k^{\rm
h}G_\Phi\|_{L^2_+}^2\,dt'\\
&\lesssim  d_k^22^{-{k}}\|\hbar^{\f12}e^\Psi
G_\Phi\|_{\wt{L}^2_{t,\dot{\tht}(t)}(\cB^{1,0})}^2,
\end{split}
\eeq
where in the last step, we used Definition \ref{def1.1}.

On the other hand,
due to \eqref{1.8bb},  \eqref{S2eq20},
\eqref{AAAP} and the support properties to the Fourier transform of
the terms in $T^{\rm h}_{\p_yu}v,$ we find
\beno
\begin{split}
\int_{0}^t&\hbar(t')\w{t'}^{-1}\bigl|\bigl(e^\Psi y\int_y^\infty \D_k^{\rm h}\bigl[T^{\rm
h}_{ \p_yu}v\bigr]_\Phi\,dy'\ |\ e^\Psi\D_k^{\rm
h}G_\Phi\bigr)_{L^2_+}\bigr|\,dt'\\
\lesssim & \sum_{|k'-k|\leq
4}\int_{0}^t\w{t'}^{-1}\|e^{-\f\Psi4}y\|_{L^2_{\rm v}}\| e^{\f34\Psi} S_{k'-1}^{\rm
h}\p_yu_\Phi(t')\|_{L^2_{\rm v}(L^\infty_\h)}\\
&\qquad\qquad\qquad\qquad\times\|\hbar^{\f12} e^{\f\Psi2}\D_{k'}^{\rm h}v_\Phi(t')\|_{L^2_+}\|\hbar^{\f12} e^{\Psi}\D_{k}^{\h}G_\Phi(t')\|_{L^2_+}\,dt'\\
\lesssim & \sum_{|k'-k|\leq
4}2^{k'}\int_{0}^t\dot{\theta}(t')\|\hbar^{\f12} e^{\Psi}\D_{k'}^{\rm h}
G_\Phi(t')\|_{L^2_+}\|\hbar^{\f12} e^{\Psi}\D_{k}^{\rm
h}G_\Phi(t')\|_{L^2_+}\,dt'.
\end{split}
\eeno
Then a similar derivation of \eqref{S4eq4} yields
\beno
\begin{split}
\int_{0}^t\hbar(t')\w{t'}^{-1}\bigl|\bigl(e^\Psi y\int_y^\infty \D_k^{\rm h}\bigl[T^{\rm
h}_{ \p_yu}v\bigr]_\Phi\,dy'\ |\ e^\Psi\D_k^{\rm
h}G_\Phi\bigr)_{L^2_+}\bigr|\,dt'\lesssim & d_k^22^{-{k}}\|\hbar^{\f12}e^\Psi
G_\Phi\|_{\wt{L}^2_{t,\dot{\tht}(t)}(\cB^{1,0})}^2.
\end{split}
\eeno

Again thanks to \eqref{S2eq20}, we get, by a similar procedure, that
\beno
\begin{split}
\int_{0}^t&\hbar(t')\w{t'}^{-1}\bigl|\bigl(e^\Psi y\int_y^\infty \D_k^{\rm h}\bigl[T^{\rm
h}_{ v}\p_yu\bigr]_\Phi\,dy'\ |\ e^\Psi\D_k^{\rm
h}G_\Phi\bigr)_{L^2_+}\bigr|\,dt'\\
\lesssim & \sum_{|k'-k|\leq
4}\int_{0}^t\w{t'}^{-1}\|e^{-\f\Psi4}y\|_{L^2_{\rm v}}\|\hbar^{\f12} e^{\f\Psi2}S_{k'-1}^{\rm h}v_\Phi(t')\|_{L^2_{\rm v}(L^\infty_\h)}\\
&\qquad\qquad\qquad\qquad\times\| e^{\f34\Psi}\D_{k'}^\h \p_yu_\Phi(t')\|_{L^2_+}\|\hbar^{\f12} e^{\Psi}\D_{k}^{\rm
h}G_\Phi(t')\|_{L^2_+}\,dt'\\
\lesssim & \sum_{|k'-k|\leq
4}2^{-\f{k'}2}\int_{0}^t\w{t'}^{-\f12}\dot{\theta}(t')\|\hbar^{\f12} e^{\f\Psi2}S_{k'-1}^{\rm h}v_\Phi(t')\|_{L^2_{\rm v}(L^\infty_\h)}\|\hbar^{\f12} e^{\Psi}\D_{k}^{\rm
h}G_\Phi(t')\|_{L^2_+}\,dt'.
\end{split}
\eeno

Yet in view of  \eqref{AAAP}, we get, by a similar derivation of
 \eqref{S4eq8}, that
 \beno
\Bigl( \int_{0}^t \w{t'}^{-1}\dot{\theta}(t')\|\hbar^{\f12} e^{\f\Psi2}S_{k'-1}^{\rm h}v_\Phi(t')\|_{L^2_{\rm v}(L^\infty_\h)}^2\,dt'\Bigr)^{\f12}
\lesssim  d_{k'}2^{\f{k'}2}\|\hbar^{\f12}e^\Psi
G_\Phi\|_{\wt{L}^2_{t,\dot{\tht}(t)}(\cB^{1,0})}.
\eeno
As a result, it comes out
\beno
\int_{0}^t\hbar(t')\w{t'}^{-1}\bigl|\bigl(e^\Psi y\int_y^\infty \D_k^{\rm h}\bigl[T^{\rm
h}_{ v}\p_yu\bigr]_\Phi\,dy'\ |\ e^\Psi\D_k^{\rm
h}G_\Phi\bigr)_{L^2_+}\bigr|\,dt'
\lesssim  d_k^22^{-{k}}\|\hbar^{\f12}e^\Psi
G_\Phi\|_{\wt{L}^2_{t,\dot{\tht}(t)}(\cB^{1,0})}^2. \eeno

Finally  again due to \eqref{S2eq20}, \eqref{AAAP} and the support properties to the
Fourier transform of the terms in $R^{\rm h}(\p_yu,\p_x\vf),$ we get, by applying Lemma \ref{lem:Bern}, that
\beno
\begin{split}
&\int_{0}^t\hbar(t')\w{t'}^{-1}\bigl|\bigl(e^\Psi y\int_y^\infty \D_k^{\rm h}\bigl[R^{\rm h}(\p_yu,\p_xG
)\bigr]_\Phi\ |\ e^\Psi\D_k^{\rm
h}G_\Phi\bigr)_{L^2_+}\bigr|\,dt'\\
&\lesssim 2^{{\f{k}2}}\sum_{k'\geq
k-3}\int_{0}^t\w{t'}^{-1}\|e^{-\f\Psi4}y\|_{L^2_{\rm v}}\|e^{\f34\Psi}\wt{\D}_{k'}^{\rm
h}\p_yu_\Phi(t')\|_{L^2_+}\\
&\qquad\qquad\qquad\times\|\hbar^{\f12}e^{\f\Psi2}{\D}^{\rm h}_{k'}v_\Phi(t')\|_{L^2_+}\|\hbar^{\f12}e^{\Psi}\D_{k}^{\rm
h}G_\Phi(t')\|_{L^2_+}\,dt'\\
&\lesssim 2^{{\f{k}2}}\sum_{k'\geq
k-3}2^{\frac{k'}2}\int_{0}^t\dot{\theta}(t')\|\hbar^{\f12} e^{\Psi}\D_{k'}^{\rm
h}G_\Phi(t')\|_{L^2_+}\|\hbar^{\f12}e^{\Psi}\D_{k}^{\rm
h}G_\Phi(t')\|_{L^2_+}\,dt',
 \end{split} \eeno
 from which and a similar derivation of \eqref{S4eq6}, we obtain
 \beno
\begin{split}
\int_{0}^t\hbar(t')\w{t'}^{-1}\bigl|\bigl(e^\Psi y\int_y^\infty \D_k^{\rm h}\bigl[R^{\rm h}(\p_yu,\p_xG
)\bigr]_\Phi\ |\ e^\Psi\D_k^{\rm
h}G_\Phi\bigr)_{L^2_+}\bigr|\,dt'\lesssim  d_k^22^{-{k}}\|\hbar^{\f12}e^\Psi
G_\Phi\|_{\wt{L}^2_{t,\dot{\tht}(t)}(\cB^{1,0})}^2.
 \end{split} \eeno

Therefore, by virtue of \eqref{pd1},  we conclude that
\beno
\int_0^t\hbar(t')\w{t'}^{-1}\bigl|\bigl(e^\Psi y\int_y^\infty \D_k^{\rm h}[\p_yuv]_\Phi
\,dy' \ |\ e^\Psi\D_k^{\rm
h} G_\Phi\bigr)_{L^2_+}\bigr|\,dt'\lesssim  d_k^22^{-{k}}\|\hbar^{\f12}e^\Psi
G_\Phi\|_{\wt{L}^2_{t,\dot{\tht}(t)}(\cB^{1,0})}^2. \eeno
This together with \eqref{S7eq13b} ensures \eqref{S7eq13a}. We thus finishes the proof of Lemma \ref{S7lem2}.
\end{proof}

\section*{Acknowledgments}

Both authors are supported by K.C.Wong Education Foundation.  M. Paicu was  partially supported by the Agence Nationale de la Recherche, Project IFSMACS, grant ANR-15-CE40-0010. P. Zhang is partially supported
by NSF of China under Grants   11371347 and 11688101,  and innovation grant from National Center for
Mathematics and Interdisciplinary Sciences.

\end{document}